\documentclass[11pt, letter]{article}%
\usepackage{amsmath}
\usepackage{amsfonts}
\usepackage{amssymb}
\usepackage{amsthm}
\usepackage{graphicx}

\newtheorem{theorem}{Theorem}[section]

\newtheorem{corollary}[theorem]{Corollary}
\newtheorem{lemma}[theorem]{Lemma}
\newtheorem{proposition}[theorem]{Proposition}

\theoremstyle{definition}
\newtheorem{definition}[theorem]{Definition}

\theoremstyle{remark}

\newtheorem{remark}[theorem]{Remark}
\numberwithin{equation}{section}
\begin{document}

\title{Analytic approximation of transmutation operators for one-dimensional stationary Dirac operators and applications to solution of initial value and spectral problems}
\author{Nelson Guti\'{e}rrez Jim\'{e}nez$^a$ and Sergii M. Torba$^b$  \\
{\small $~^a$ Instituto de Matem\'{a}ticas, Facultad de Ciencias Exactas y Naturales, Universidad de Antioquia,}\\{\small Calle 67 No.~53--108, Medell\'{i}n, COLOMBIA}
\\{\small $~^b$ Departamento de Matem\'{a}ticas, CINVESTAV del IPN, Unidad
Quer\'{e}taro, }\\{\small Libramiento Norponiente No. 2000, Fracc. Real de Juriquilla,
Quer\'{e}taro, Qro. C.P. 76230 MEXICO}\\{\small e-mail:
njgutier@gmail.com, storba@math.cinvestav.edu.mx\thanks{The authors acknowledge the support from CONACYT, Mexico via the project 222478. N.~Guti\'{e}rrez Jim\'{e}nez would like to express his gratitude to the Mathematical department of Cinvestav where he completed the PhD program (the presented paper contains part of the obtained results) and to CONACYT, Mexico for the financial support making it possible.}}}

\maketitle

\begin{abstract}
A method for approximate solution of initial value and spectral problems for one dimensional Dirac equation based on an analytic approximation of the transmutation operator is presented. In fact the problem of numerical approximation of solutions
is reduced to approximation of the potential matrix by a finite linear combination of matrix valued functions related to
generalized formal powers introduced in \cite{KTG1}. Convergence rate estimates in terms of smoothness of the potential are proved.
The method allows one to compute both lower and higher eigendata with an extreme accuracy.
\end{abstract}

\section{Introduction}
In the present paper we consider a one-dimensional Dirac equation, that is, the system of linear differential equations of the form
\begin{equation}\label{DiracSystem}
    \begin{cases}
        y_2'+p(x) y_1 + q(x) y_2 = \lambda y_1,\\
        -y_1'+q(x) y_1 - p(x) y_2 = \lambda y_2,
    \end{cases}
\end{equation}
or in the matrix form,
\begin{equation}\label{DiracMatrix}
B \frac{dY}{dx} + Q(x)Y = \lambda Y,\qquad Y(x) = \begin{pmatrix}
    y_1(x) \\
    y_2(x) \\
\end{pmatrix},
\end{equation}
where
\begin{equation}\label{Eq DefMatricesBQ}
B = \begin{pmatrix}
    0 & 1 \\
    -1 & 0 \\
\end{pmatrix},\qquad Q(x) =
\begin{pmatrix}
    p(x) & q(x) \\
    q(x) & -p(x) \\
\end{pmatrix},
\end{equation}
$p, q\in L_1[0,b]$ are given complex-valued functions of the real variable $x$, and $\lambda$ (called the spectral parameter) is an arbitrary complex constant.

Due to its importance, the one-dimensional stationary Dirac equation has been the object of study in various areas of mathematics and mathematical physics (\cite{AHrynivM}, \cite{ClarkGesztesy2001}, \cite{Hryniv2011}, \cite{SavchukShkalikov}). And since the discovery of the fact that the one-dimensional Dirac equation appears during the solution of the modified Korteweg-de Vries equation by the inverse scattering method \cite{AKNS}, \cite{AblowitzSegur} there is a strong interest in solution of direct and inverse spectral problems related to \eqref{DiracSystem}. For an introduction to the Dirac equation see, for example, \cite{BG1}, \cite{BG2}, \cite{WGreiner}, \cite{LevitanSargsjan}, \cite{TB}. Up to our best knowledge, there are only several papers \cite{AnnabyTharwat2007}, \cite{AnnabyTharwat}, \cite{AnnabyTharwat2013}, \cite{Tharwat}, where a spectral problem for one-dimensional Dirac equation is solved numerically using a sampling method. All these approaches share a common disadvantage, a truncated representation works only in some neighborhood of zero and requires significant computation time due to necessity to solve particular initial value problems for all sampling points.

We are looking for a method of approximate solution of \eqref{DiracSystem} which can be efficiently used to obtain large sets of eigenvalues and eigenfunctions of spectral problems associated with \eqref{DiracSystem}. That is, we are interested in a method where some coefficients may need to be precomputed, but afterwards obtaining an approximate value of $Y(x,\lambda)$ for each additional $\lambda$ can be done in almost no cost. Additionally, we would like to have an error bound uniform for all $\lambda\in\mathbb{R}$ and exponentially decreasing as a function of a parameter $N$ describing the approximation. In \cite{KTG1} we presented a representation for the solution of \eqref{DiracSystem} in the form of spectral parameter power series allowing fast computation of approximation $Y_N(x,\lambda)$ but whose accuracy deteriorates rapidly as $\lambda\to\infty$.

In \cite{KrT2015} we proposed a method of approximate solution of one-dimensional stationary Schr\"{o}dinger equation based on an approximation of the transmutation operator. The method possesses all sought-for properties: efficient evaluation for each additional $\lambda$, uniform error bound and exponentially fast convergence. See \cite{KrT2016}, \cite{KrMorT2016}, \cite{KrST}, \cite{KTNavarro}, \cite{KTRaul}, \cite{KShT2018}, \cite{KrT2018}, \cite{KrT2020} for further development. The aim of the present paper is to extend the method to the system \eqref{DiracSystem}.

A transmutation operator for the Dirac equation can be realized as a Volterra integral operator. The integral kernel of this operator satisfies a certain Goursat problem for hyperbolic matrix equation. Using the generalized formal powers introduced in \cite{KTG1} we construct a complete system of solutions for this hyperbolic matrix equation (in the sense that any solution can be approximated in the uniform norm by a finite linear combination of solutions from the complete system), see Section \ref{Sect4}. In order to find coefficients of a linear combination approximating the integral kernel we utilize the Goursat data, see Section \ref{Sect5}. We prove the decay rate estimates depending on the smoothness of the potential, see Section \ref{Sect6}. Finally, we show how constructed approximation for the integral kernel leads to an efficient approximation of the solutions of \eqref{DiracSystem}, see Section \ref{Sect7}. We illustrate the results with several numerical examples in Section \ref{Sect9}. 

We would like to mention that the developed theory also provides different view on the results from \cite{KrT2015} and \cite{KrT2016}. In particular, we do not need an inverse of the transmutation operator in the proofs. Moreover, a one dimensional stationary Schr\"{o}dinger equation from \cite{KrT2015} can be transformed into a Dirac equation. Applying the proposed method for this Dirac equation, we obtained an analytic approximation of transmutation operator which is different to those of \cite{KrT2015}. Corresponding results are presented in Section \ref{Sect8}.

In Appendices \ref{Append1}--\ref{Append5} we prove several technical results on well-posedness of Goursat and Cauchy problems for the hyperbolic equation satisfied by the integral kernel $K$, study the smoothness of the integral kernel $K$ (depending on the smoothness of the potential matrix $Q$) and present explicit formula for the solution of a least squares minimization problem.

Throughout the paper we use the notation $|A|$ for matrix norm, from which we require to be submultiplicative, that is, to satisfy $|AB|\le |A|\cdot |B|$. From several possible norms we chose the one induced by the matrix scalar product $\langle A, B\rangle = \operatorname{tr} (AB^\ast)$, where $\operatorname{tr}A$ is the trace and $A^\ast$ denotes the conjugate transpose of a matrix $A$. This norm is known as the
Frobenius norm, $\|A\|^2=\sum_{i=1}^m\sum _{j=1}^{n}|a_{ij}|^2$. Also we denote by $C\bigl([a,b],\mathcal{M}_2\bigr)$ the space of continuous $2\times 2$-matrix functions $F$ equipped with the norm $\|F\|_{C\bigl([a,b],\mathcal{M}_2\bigr)} = \max_{[a,b]}|F|$ and by $L_2((a,b),\mathcal{M}_2)$ the space of integrable $2\times 2$-matrix functions equipped with the scalar product
\[
\langle F, G\rangle_{L_2((a,b),\mathcal{M}_2)} = \int_a^b \langle F(x), G(x)\rangle\,dx =
\int_a^b \operatorname{tr} \left(F(x) G^\ast(x)\right)\,dx.
\]
If there is no ambiguity, we will simply write $C[a,b]$ and $L_2(a,b)$ for the spaces and $\|F\|$ for the norm.

\section{Transmutation operators}\label{Sect2}

Following Levitan \cite{Levitan}, let $E$ be a linear topological space and $E_{1}$ its linear subspace (not necessarily closed). Let $\mathcal A_{1},\mathcal A_{2}:E_{1}\rightarrow E$ be linear operators.

\begin{definition}\label{DefTransOpe} A linear invertible operator $T$ defined on the whole $E$ such that $E_{1}$ is invariant under the action of $T$ is called a transmutation operator for the pair of operators $\mathcal A_{1}$ and $\mathcal A_{2}$ if it fulfills the following two conditions.
\begin{enumerate}
	\item Both the operator $T$ and its inverse $T^{-1}$ are continuous in $E$;
	\item The following operator equality is valid  $$\mathcal A_{1}T=T\mathcal A_{2}$$ or which is the same $$\mathcal A_{1}=T\mathcal A_{2}T^{-1}.$$
\end{enumerate}
\end{definition}

Let us denote by
\begin{equation}\label{ExpressionDiffDiracOperators}
    \mathcal A_{Q}:=B\frac{d}{dx}+Q(x),
\end{equation}
a differential operator related to the system \eqref{DiracMatrix},
and by $\mathcal A_{0}$ the differential operator \eqref{ExpressionDiffDiracOperators} having $Q$ as the null matrix-valued function.

Unless otherwise stated, in this section we assume that $Q$ is a continuously differentiable matrix-valued function on $\left[0,b\right]$. We will denote the class of such functions by $C^1([0,b],\mathcal{M}_2)$. Here we would like to mention that the approximation constructed further in this paper requires $Q$ to be defined only on $[0,b]$, but the definition of the transmutation operator either requires the symmetric segment $[-b,b]$ or some initial condition at 0 (which we would like to avoid). Since neither the expression defining the transmutation operator nor the final result depend on the values of the potential outside of the segment $[0,b]$, we will consider such extended potential when necessary in the paper.

Let $E$ be the space 
\[
C(\left[-b,b\right],\mathbb{C}^{2})=\left\{Y(t)=(y_{1}(t),y_{2}(t))^{T}\,|\, y_{1},y_{2}\in C[-b,b] \right\}
\]
and let the operators $\mathcal A_{0}$ and $\mathcal A_{Q}$ act on the space
\[
E_1 = C^{1}(\left[-b,b\right],\mathbb{C}^{2})=\left\{Y(t)=(y_{1}(t),y_{2}(t))^{T}\,|\, y_{1},y_{2}\in C^{1}[-b,b] \right\}.
\]

The following result holds.
\begin{theorem}\label{Th DiracTransOpe}
Suppose  that $Q$ is a continuously differentiable matrix-valued function on $\left[-b,b\right]$. Then a transmutation operator $T$, relating the operators $\mathcal A_{0}$ and $\mathcal A_{Q}$ in the sense of Definition \ref{DefTransOpe} for all $Y\in C^{1}(\left[-b,b\right],\mathbb{C}^{2})$, can be realized in the form of a Volterra integral operator
\begin{equation}\label{Eq VolIntOperator}
TY(x)=Y(x)+\int_{-x}^{x}K(x,t)Y(t)\,dt,
\end{equation}
where $K(x, t)$ is a $2\times 2$ matrix-valued function satisfying the partial differential equation
\begin{equation}\label{Eq IntKernel}
BK_{x}(x,t)+K_{t}(x,t)B = -Q(x)K(x,t)
\end{equation}
with the Goursat conditions
\begin{align}
BK(x,x)-K(x,x)B &=-Q(x)\label{Eq GsatDataX+},\\
BK(x,-x)+K(x,-x)B&= 0\label{Eq GsatDataX-}.
\end{align}

Conversely, if $K(x,t)$ is a solution of the Goursat problem \eqref{Eq IntKernel}--\eqref{Eq GsatDataX-}, then the
operator $T$ determined by formula \eqref{Eq VolIntOperator} is a transmutation operator for the pair of operators  $\mathcal A_{Q}$ and $\mathcal A_{0}$.
\end{theorem}

The proof is similar to those of \cite[Theorem 10.3.1]{LevitanSargsjan}, the only part requiring modification is the proof of existence and uniqueness of the solution of the Goursat problem \eqref{Eq IntKernel}, \eqref{Eq GsatDataX+},  \eqref{Eq GsatDataX-}. The main reason is that the proof from \cite{LevitanSargsjan} requires continuation of the potential to the whole axis followed by solving two integral equations. Instead, we transformed the Goursat problem to an equivalent integral equation similar to \cite[(Section 1.2, problem 5)]{VMchenko} and adapted the proof from \cite{VMchenko}. Please see Appendix \ref{Append1} for details.

Note that for the definition of the transmutation operator \eqref{Eq VolIntOperator} one requires knowledge of the integral kernel $K$ only in the union $\Omega$ of the sets
\begin{equation}\label{Omega pm}
\Omega^{+}=\left\{(x,t): 0\le x\le b,\, \left|t\right|\le x\right\}\qquad\text{and}\qquad \Omega^{-}=\left\{(x,t):-b< x\le 0,\, \left|t\right|\le \left|x\right|\right\}.
\end{equation}
The Goursat problem \eqref{Eq IntKernel}--\eqref{Eq GsatDataX-} can be solved independently in the domains $\Omega^{+}$ and $\Omega^{-}$.

\begin{remark}
Under the condition of the potential to be continuously differentiable the Goursat problem \eqref{Eq IntKernel}--\eqref{Eq GsatDataX-} possesses a classical solution, i.e., the solution is a differentiable function satisfying equation \eqref{Eq IntKernel} in every point. It is well known that such restriction can be weakened. For potentials belonging to $L_1((0,b),\mathcal{M}_2)$ it is sufficient to ask that $K$ satisfies an equivalent integral equation and, when necessary, approximate $Q$ by a sequence of continuously differentiable potentials and pass to the limit, see \cite[Theorem 2.1]{AHrynivM}, \cite{LevitanSargsjan}, \cite{VMchenko}, \cite{Hryniv2011} for details. Further in this paper, when only the existence of the continuous operator $T$ and its property to map solution into solution are necessary, we will formulate the results for larger class of potentials.
\end{remark}

Consider the equation $\mathcal{A}_0 V = \lambda V$, $V=(v_1, v_2)^T$. Its general solution is given by
\begin{equation}\label{GenSolHomogen}
V(x)=
\begin{pmatrix}{v_{1}(x)}\\
               {v_{2}(x)}
\end{pmatrix}=c_{1}
\begin{pmatrix}{\cos \lambda x}\\
               {\sin \lambda x}
\end{pmatrix}+c_{2}
\begin{pmatrix}{-\sin\lambda x}\\
               {\cos\lambda x}
\end{pmatrix},\quad c_{1}, c_{2}\in \mathbb{C}.
\end{equation}
If one knows the transmutation operator for the pair $\mathcal{A}_0$, $\mathcal{A}_Q$, then the vector-valued function
\begin{equation}\label{Transmut SolToSol}
Y = TV
\end{equation}
is a solution of the Dirac equation \eqref{DiracSystem}.

Unfortunately, transmutation operator is known in an explicit form only for few potentials. However, suppose we can find an approximate integral kernel $K_N$ in a form
\begin{equation*}
K_{N}(x, t) =\sum_{n=0}^{N}\mathcal{K}_n(x) t^{n},
\end{equation*}
where $\mathcal{K}_{n}$ are $2\times 2$ matrix-valued functions. Then we may approximate the solutions of the system \eqref{DiracSystem} by
the vector-valued functions
\begin{equation*}
C_{N}(x,\lambda)=\begin{pmatrix}{\cos\lambda x}\\
                        {\sin\lambda x}
					\end{pmatrix}
		+\sum_{n=0}^{N}\mathcal{K}_{n}(x)\int_{-x}^{x}
		\begin{pmatrix}{t^{n}\cos\lambda t}\\
                   {t^{n}\sin\lambda t}
    \end{pmatrix}\,dt
\end{equation*}
and
\begin{equation*}
S_{N}(x,\lambda)=\begin{pmatrix}{-\sin\lambda x}\\
                        {\cos\lambda x}
			    \end{pmatrix}
		+\sum_{n=0}^{N}\mathcal{K}_{n}(x)\int_{-x}^{x}
		\begin{pmatrix}{-t^{n}\sin\lambda t}\\
                   {t^{n}\cos\lambda t}
		\end{pmatrix}\,dt.
\end{equation*}
Note that the integrals can be easily evaluated in the closed form.
The error of approximation can be estimated uniformly with respect to $\lambda\in\mathbb{R}$. Indeed, let
\begin{equation}\label{Estimate K KN}
\varepsilon(x)=\left\|K(x,\cdot)-K_N(x,\cdot)\right\|_{L_{2}(-x,x)}.
\end{equation}
Then using the Cauchy-Schwarz inequality we obtain
\begin{equation}\label{Estimate for Solution}
\begin{split}
\left|T\begin{pmatrix}{\cos\lambda x}\\
                       {\sin\lambda x}
				\end{pmatrix}-C_{N}(x,\lambda)\right|
&=\left|
\int_{-x}^{x}
\left(K(x,t)-K_{N}(x,t)\right)\begin{pmatrix} \cos\lambda t\\
                                         \sin\lambda t
                          \end{pmatrix}
													dt\right|\\
& = \left| \int_{-x}^x \begin{pmatrix} (K_{11} - K_{N,11})(x,t) \cos\lambda t + (K_{12} - K_{N,12})(x,t) \sin\lambda t\\
(K_{21} - K_{N,21})(x,t) \cos\lambda t + (K_{22} - K_{N,22})(x,t) \sin\lambda t
\end{pmatrix}dt\right|\\
&\le
\sqrt 2\varepsilon(x)\biggl(\int_{-x}^{x}\bigl(\cos^{2}\lambda t+\sin^{2}\lambda t\bigr)\,dt\biggr)^{1/2} \le
2\sqrt{x}\varepsilon(x).
\end{split}
\end{equation}
Similarly for the second solution.

We utilize the following scheme to construct an approximation $K_N$.
\begin{itemize}
\item Even though the transmutation operator can not be obtained explicitly, we show that by using a simple recurrent integration procedure one can obtain the images of the powers of $x$, i.e., the equalities
\begin{equation*}
T\begin{pmatrix}
{x^{k}}\\
{0}
\end{pmatrix}
=\Phi_{k}(x)
\qquad\text{and}\qquad
T\begin{pmatrix}{0}\\
               {x^{k}}
\end{pmatrix}
=\Psi_{k}(x),\quad k=0,1,2,\ldots.
\end{equation*}

\item There exist matrix-valued functions $\mathcal O_{m}^{j}$, $j\in{0,1,2,3}$, $m\in\mathbb{N}_0$ (obtained as images of so called wave polynomial matrices under the action of transmutation operator $T$) which form a complete system of solutions for equation \eqref{Eq IntKernel} in the sense that any solution of \eqref{Eq IntKernel} can be approximated by a finite linear combination of the functions $\mathcal{O}_m^j$. Moreover, each function $\mathcal{O}_m^j$ is a polynomial in $t$ whose coefficients are the functions $\Phi_k$ and $\Psi_k$ multiplied by some binomial coefficients.

\item We look for an approximate integral kernel in the form
\[
K_N(x,t) = \sum_{n=0}^N\Bigl(a_{n}\mathcal O_{n}^{1}(x,t) +b_{n}\mathcal O_{n}^{2}(x,t)+c_{n}\mathcal O_{n}^{3}(x,t)+d_{n}\mathcal O_{n}^{4}(x,t)\Bigr).
\]
Since each $\mathcal O_{m}^{j}$ is a solution of \eqref{Eq IntKernel}, the function $K_N$ is also a solution of \eqref{Eq IntKernel}.

\item The coefficients of the approximation
are obtained from the Goursat data \eqref{Eq GsatDataX+}, \eqref{Eq GsatDataX-}. Moreover, we show that by taking the half sum of the Goursat data, it is sufficient to solve only one minimization problem, cf. \cite[Theorem 5.1]{KrT2015}. By using the least squares method, we reduce the problem of finding the coefficients to solution of two linear systems of equations.

\item  An estimate for \eqref{Estimate K KN} immediately follows from the well-posedness of the Goursat problem.

\item Finally, we show that if $Q$ is a matrix-valued function of class $C^{r}$ on $\left[0,b\right]$,
the function $\varepsilon(x)$ in \eqref{Estimate K KN} can be bounded by $C_{r}/N^{r}$ for all $N>r$, where the constant $C_{r}$ depends on $Q$ and $x$ and does not depend on $N$.

\end{itemize}

\section{Recurrence integrals, the SPPS representation and mapping property}\label{Sect3}

\subsection{Spectral parameter power series}\label{SubSectSPPS}
Following \cite{KTG1}, let $(f,g)^T$ be a solution of the homogeneous Dirac equation $B\frac{dY}{dx}+Q(x)Y=0$, i.e.,
\begin{align}
\displaybreak[2]
        g'+p(x) f + q(x) g &=0\label{Eq1 HomDiracsystem},\\
        -f'+q(x) f - p(x) g &=0\label{Eq2 HomDiracsystem},
\end{align}
and suppose that both functions $f$ and $g$ are non-vanishing on $[0,b]$ (see Remark \ref{Rmk Non vanishing} with respect to the existence of such solutions).
Suppose additionally that $f(0)g(0)=1$. Consider the following systems of functions defined by recurrence relations:
\begin{align}
    X^{(0)}(x)&=-\int_{0}^{x}\dfrac{p(s)}{f^{2}(s)}ds,\qquad Y^{(0)}(x)=1 +\int_{0}^{x}\dfrac{p(s)}{g^{2}(s)}ds, \label{X0Y0}\\
\displaybreak[2]
    Z^{(n)}(x)&=\int_{0}^{x}\Bigl(f^{2}(s)X^{(n)}(s)+g^{2}(s)Y^{(n)}(s)\Bigr)ds,                          \label{Zn}\\
\displaybreak[2]
    X^{(n+1)}(x)&=-(n+1)\int_{0}^{x}\Bigl(\dfrac{p(s)}{f^{2}(s)}Z^{(n)}(s)+\dfrac{g(s)}{f(s)}Y^{(n)}(s)\Bigr)ds,  \label{Xn}\\
\displaybreak[2]
    Y^{(n+1)}(x)&=(n+1)\int_{0}^{x}\Bigl(\dfrac{p(s)}{g^{2}(s)}Z^{(n)}(s)+\dfrac{f(s)}{g(s)}X^{(n)}(s)\Bigr)ds,\qquad  n=0, 1,2,\ldots\label{Yn}
\end{align}
Similarly we use as the initial functions
\begin{equation}\label{XtYt0}
    \widetilde X^{(0)}(x) = 1 +\int_{0}^{x}\dfrac{p(s)}{f^{2}(s)}\,ds, \qquad    \widetilde Y^{(0)}(x) = -\int_{0}^{x}\dfrac{p(s)}{g^{2}(s)}\,ds \end{equation}
and define the functions $\widetilde Z^{(n)}$, $\widetilde X^{(n)}$ and $\widetilde Y^{(n)}$, $n\ge 0$ by formulas \eqref{Zn}--\eqref{Yn} replacing $X^{(n)}$, $Y^{(n)}$ and $Z^{(n)}$ by $\widetilde X^{(n)}$, $\widetilde Y^{(n)}$ and $\widetilde Z^{(n)}$, correspondingly.

The following result obtained in \cite{KTG1} establishes the relation of the
systems of functions $\{X^{(n)}\}_{n=0}^\infty$, $\{Y^{(n)}\}_{n=0}^\infty$, $\{\widetilde X^{(n)}\}_{n=0}^\infty$ and $\{\widetilde Y^{(n)}\}_{n=0}^\infty$ to the Dirac equation.

\begin{theorem}[\cite{KTG1}]\label{Th DiracSPPS} Suppose that both functions  $f$ and $g$ are absolutely continuous, non-vanishing on $[0,b]$ and satisfy the homogeneous Dirac equation \eqref{Eq1 HomDiracsystem}--\eqref{Eq2 HomDiracsystem} a.e.\ on $[0,b]$. Assume that $\frac{p}{f^{2}}$, $\frac{p}{g^{2}}$, $\frac{f}{g}$ and $\frac{g}{f}$ are integrable functions and that $\lambda$ is an arbitrary complex parameter.
Then the general  solution of the Dirac equation \eqref{DiracSystem} has the form
\begin{equation*}
Y = c_1 Y_1 + c_2 Y_2 = c_1\begin{pmatrix}
    u_1 \\
    v_1 \\
\end{pmatrix} + c_2\begin{pmatrix}
    u_2 \\
    v_2 \\
\end{pmatrix},
\end{equation*}
where $c_1$ and $c_2$ are arbitrary complex constants and
\begin{equation}\label{DiracSPPS}
\begin{pmatrix}
    u_1 \\
    v_1 \\
\end{pmatrix} = \sum_{n=0}^\infty \frac{\lambda^{n}}{n!} \begin{pmatrix} {f\widetilde{X}^{(n)}}\\
                                                         {g\widetilde{Y}^{(n)}}
																					\end{pmatrix}\qquad\text{and}\qquad
\begin{pmatrix}
    u_2 \\
    v_2 \\
\end{pmatrix}= \sum_{n=0}^\infty \frac{\lambda^{n}}{n!} \begin{pmatrix} {fX^{(n)}} \\
		                                                     {gY^{(n)}}
																				 \end{pmatrix}.
\end{equation}
Both series converge uniformly with respect to $x\in[0,b]$ and with respect to $\lambda$ belonging to a compact set on a complex plane to vector-valued functions $(u_{i},v_{i})^{T}$, $u_{i},v_{i}\in AC[0,b]$, $i=1,2.$
The solutions $Y_{1}$ and $Y_{2}$ satisfy the following initial conditions:
\begin{equation}\label{SPPS IC}
Y_{1}(0)=\begin{pmatrix}f(0)\\
                                           0
                             \end{pmatrix},
\qquad
Y_{2}(0)=\begin{pmatrix}0\\
                                              g(0)
                               \end{pmatrix}.														
\end{equation}
\end{theorem}

Representation \eqref{DiracSPPS}, also known as the SPPS method (Spectral Parameter Power Series), present an efficient
and highly competitive technique for solving a variety of spectral and scattering problems related to Dirac
equation. The first work implementing an analogue of Theorem \ref{Th DiracSPPS} for numerical solution of Sturm-Liouville spectral problems was \cite{KrPorter2010} and later on the SPPS method was used in a number of publications (see \cite{KKRosu}, \cite{KTG1} and references therein).

\begin{remark}\label{Rmk Non vanishing}
It is worth mentioning that the existence and construction of the required solution $(f,g)^T$ presents no
difficulty. Indeed, let $p$ and $q$ be real valued and continuous on $[0, b]$. Then \eqref{Eq1 HomDiracsystem}--\eqref{Eq2 HomDiracsystem} possesses two linearly independent real-valued
solutions $(f_1,g_1)^T$  and $(f_2,g_2)^T$ such that neither $f_1$ and $f_2$ nor $g_1$ and $g_2$ can have common zero. Thus, one may choose $(f,g)^T = (f_1,g_1)^T + i(f_2,g_2)$. Moreover, for the construction of $(f_1,g_1)^T$  and $(f_2,g_2)^T$ the same SPPS method may be used, see \cite[Section 2.3]{KTG1} for details. In the case of complex-valued coefficients the existence of a non-vanishing
solution was shown in \cite[Proposition 2.9]{KTG1}.
\end{remark}

\subsection{Mapping property}

Unfortunately the integral kernel of the operator $T$ can be found in closed form only for a few particular potentials, in general it is unknown. So there is no way to determine the result of $T$ acting on an arbitrary vector-valued function. However, it is possible to determine the result of $T$ acting on an arbitrary vector function of the form $(x^{k}, x^{m})^{T}$, and, hence, on arbitrary vector-function $\left(p_{1}, p_{2}\right)^{T}$, where $p_{1}$ and $p_{2}$ are polynomials.

Under the assumptions of Subsection \ref{SubSectSPPS}, recalling that $f(0)g(0)=1$, the following mappings are valid, c.f., \eqref{GenSolHomogen}, \eqref{Transmut SolToSol}, \eqref{Eq VolIntOperator} and \eqref{SPPS IC}.
\begin{equation}\label{Eq1 SPPSandTransOpe}
T\begin{pmatrix}{\cos\lambda x}\\
                {\sin\lambda x}
 \end{pmatrix}=\frac{1}{f(0)}Y_1(x) =g(0)Y_1(x)=
g(0)\sum_{k=0}^{\infty}\frac{\lambda^{k}}{k!}\begin{pmatrix}{f(x)\widetilde{X}^{(k)}(x)}\\
                                                            {g(x)\widetilde{Y}^{(k)}(x)}
																					    \end{pmatrix},
\end{equation}
and
\begin{equation}\label{Eq2 SPPSandTransOpe}
T\begin{pmatrix}{-\sin\lambda x}\\
                {\cos\lambda x}
	\end{pmatrix}=\frac{1}{g(0)}Y_2(x) = f(0)Y_2(x)=
	f(0)\sum_{k=0}^{\infty}\frac{\lambda^{k}}{k!}\begin{pmatrix}{f(x)X^{(k)}(x)}\\
	                                                            {g(x)Y^{(k)}(x)}
																							  \end{pmatrix}.
\end{equation}
Furthermore, the solutions of the Dirac equation are analytic in the spectral parameter. Therefore, by representing $\sin\lambda x$ and $\cos\lambda x$ as power series, and by comparing coefficients near the powers of $\lambda$ in \eqref{Eq1 SPPSandTransOpe} and \eqref{Eq2 SPPSandTransOpe}, we obtain the above assertion.

We introduce the infinite sequences  of vector-valued functions $\left\{\Phi_{k}\right\}_{k=0}^{\infty}$ and $\left\{\Psi_{k}\right\}_{k=0}^{\infty}$ given respectively by
\begin{equation}\label{Phisubk}
\Phi_{k}=
\begin{cases}
(-1)^{(k+1)/2}g(0)\begin{pmatrix}{fX^{(k)}}\\{gY^{(k)}}\end{pmatrix}, & k\,\text{odd},\\
(-1)^{k/2}f(0)\begin{pmatrix}{f\widetilde{X}^{(k)}}\\{g\widetilde{Y}^{(k)}}\end{pmatrix}, & k\,\text{even}
\end{cases}
\end{equation}
and
\begin{equation}\label{Psisubk}
\Psi_{k}=
\begin{cases}
(-1)^{(k-1)/2}g(0)\begin{pmatrix}{f\widetilde{X}^{(k)}}\\{g\widetilde{Y}^{(k)}}\end{pmatrix}, & k\,\text{odd},\\
(-1)^{k/2}f(0)\begin{pmatrix}{fX^{(k)}}\\{gY^{(k)}}\end{pmatrix}, & k\,\text{even}.
\end{cases}
\end{equation}

Finally, we obtain the following theorem.

\begin{theorem}[Mapping theorem]\label{Th MappingTheorem} Let $p,q\in L^{1}\left(0,b\right)$ be complex valued functions. Let $f$ and $g$ be as in Theorem \ref{Th DiracSPPS} normalized according to the condition $f(0)g(0)=1$. Let $T$ be the transmutation operator for $\mathcal A_{0}$ and $\mathcal A_{Q}$, and let $\Phi_{k}$ and $\Psi_{k}$ be vector-valued functions defined by \eqref{Phisubk} and \eqref{Psisubk} respectively. Then
\begin{equation}\label{Eq MappingTheorem}
T\begin{pmatrix}
{x^{k}}\\
{0}
\end{pmatrix}
=\Phi_{k}(x)
\qquad\text{and}\qquad
T\begin{pmatrix}{0}\\
               {x^{k}}
\end{pmatrix}
=\Psi_{k}(x),\quad k=0,1,2,\ldots.
\end{equation}
\end{theorem}

\section{Generalized wave polynomials: a complete system of solutions of \eqref{Eq IntKernel}}\label{Sect4}

Following ideas from \cite[Section 4]{KKTT} let us first consider the simplest hyperbolic equation of the form \eqref{Eq IntKernel} having $Q\equiv 0$,
\begin{equation}\label{Eq IntKernel0}
Bk_{x}(x,t)+k_{t}(x,t)B =0.
\end{equation}

It is easy to see that having a solution $k$ of equation \eqref{Eq IntKernel0} in the square $[-b,b]\times [-b,b]$, the function
\[
\tilde K = T[k]
\]
is a solution of \eqref{Eq IntKernel} in the same square $[-b,b]\times [-b,b]$.
Here the operator $T$ acts with respect to the variable $x$, i.e.,
\[
\tilde K(x,y) = T[k](x,y) = k(x,y) + \int_{-x}^x K(x,t) k(t,y)\,dt.
\]
Let us study first the solutions of \eqref{Eq IntKernel0}.

Let $\mathcal{H}=AC([-b,b],\mathcal{M}_2)$, i.e., is the space of absolutely continuous $2\times 2$ matrix-valued functions.
Consider the operators $\mathcal{P}^{+}$ and $\mathcal{P}^{-}$ acting on $\mathcal{H}$ as follows
\begin{equation}\label{Eq Projectors}
\mathcal{P}^{-}\left[A\right]:=\dfrac{1}{2}\left(A-BAB\right)\quad\text{and}\quad \mathcal{P}^{+}\left[A\right]:=\dfrac{1}{2}\left(A+BAB\right),
\end{equation}
where $A\in\mathcal H$ and $B$ is the matrix in \eqref{Eq DefMatricesBQ}.

The operators $\mathcal{P}^{-}$ and $\mathcal{P}^{+}$ are projectors and decompose the space  $\mathcal{H}$ in the direct sum of the spaces $\mathcal{H}^{-}:=\ker \mathcal{P}^{-}$ and $\mathcal{H}^{+}:=\ker \mathcal{P}^{+}$. $\mathcal{H}^{-}$ corresponds to the subspace of matrix-valued functions that anti-commute with $B$ and $\mathcal{H}^{+}$ corresponds to the subspace of matrix-valued functions that commute with $B$ . It should be noted that $-B^{2}$ equals the identity matrix and $Q\in \mathcal{H}^{-}$. 

Note that the projectors $\mathcal{P}^{+}$ and $\mathcal{P}^{-}$ can be applied to any $2\times 2$ matrix valued function as well. Later in the paper, for $X$ being a space of matrix valued functions, we will use the notations $X^+$ and $X^-$ for the subspaces of functions commuting and anti-commuting with $B$, respectively.

\begin{proposition}\label{lm GenSolForIntKernelEq0} The general solution of  equation \eqref{Eq IntKernel0} has the following form
\begin{equation}\label{Eq1 GenSolForIntKernelEq0}
k(x,t)=\mathcal{P}^{+}[h_{1}]\left(\frac{x+t}{2}\right)+\mathcal{P}^{-}[h_{2}]\left(\frac{x-t}{2}\right),
\end{equation}
where $h_{1}$ and $h_{2}$ are arbitrary absolutely continuous functions in $\mathcal{H}$.
\end{proposition}

\begin{proof}
An easy computation shows that the right hand side of \eqref{Eq1 GenSolForIntKernelEq0} satisfy \eqref{Eq IntKernel0}. On the other hand, let $k$ be a solution of equation \eqref{Eq IntKernel0}. Define $h(\xi(x,t),\eta(x,t))=k(x,t)$ via the change of coordinates given by $\xi=(x+t)/2$ and $\eta=(x-t)/2$. It follows that $2k_x=(h_\xi+h_\eta)$ and $2k_t=(h_\xi-h_\eta)$. Substituting these into \eqref{Eq IntKernel0} yields $\mathcal{P}^{-}\left[h_{\xi}\right](\xi,\eta)+\mathcal{P}^{+}\left[h_{\eta}\right](\xi,\eta)=0$. Applying the projectors $\mathcal{P}^{+}$ and $\mathcal{P}^{-}$ we have $\mathcal{P}^{+}\left[h_{\eta}\right](\xi,\eta)=0$ and $\mathcal{P}^{-}\left[h_{\xi}\right](\xi,\eta)=0$. From the last equalities, integrating with respect to the variables $\eta$ and $\xi$ we obtain that $\mathcal{P}^{+}\left[h\right](\xi,\eta)=c_{1}(\xi)$ and $\mathcal{P}^{-}\left[h\right](\xi,\eta)=c_{2}(\eta)$ for some $c_{1}\in\mathcal{H}^{-}$ and $c_{2}\in\mathcal{H}^{+}$. Thus,
\[
k(x,t)=\mathcal{P}^{+}\left[h\right](\xi,\eta)+\mathcal{P}^{-}\left[h\right](\xi,\eta)= c_{1}\left(\frac{x+t}{2}\right)+c_{2}\left(\frac{x-t}{2}\right).
\]
Finally, we have $c_{1}=\mathcal{P}^{+}\left[h_{1}\right]$ and $c_{2}=\mathcal{P}^{-}\left[h_{2}\right]$, for some $h_{1},h_{2}\in \mathcal{H}$, because the image of $\mathcal{P}^{+}$ is $\mathcal{H}^{-}$ and the image of $\mathcal{P}^{-}$ is $\mathcal{H}^{+}$, which completes the proof.
\end{proof}

\begin{remark} Since $\mathcal{H}^{-}=\ker \mathcal{P}^{-}$ and $\mathcal{H}^{+}=\ker \mathcal{P}^{+}$ we can consider as a general solution of equation \eqref{Eq IntKernel0} the formula
\begin{equation}\label{Eq2 GenSolForIntKernelEq0}
k(x,t)=\mathcal{P}^{+}[h]\left(x+t\right)+\mathcal{P}^{-}[h]\left(x-t\right),
\end{equation}
where $h$ is arbitrary absolutely continuous matrix-valued function. We observe that if we substitute $h(x)=\mathcal{P}^{+}h_{1}(x/2)+\mathcal{P}^{-}h_{2}(x/2)$ into \eqref{Eq2 GenSolForIntKernelEq0} we recover \eqref{Eq1 GenSolForIntKernelEq0}.
\end{remark}

We introduce the system of wave matrices $\left\{P_{m}^{i}, i=1,2,3,4\right\}_{m=0}^{\infty}$  as a result of applying  formula \eqref{Eq2 GenSolForIntKernelEq0} to four matrix-valued functions
\[
\begin{pmatrix}{x^{m}}&{0}\\{0}&{0}\end{pmatrix},\,\begin{pmatrix}{0}&{x^{m}}\\{0}&{0}\end{pmatrix},\,\begin{pmatrix}{0}&{0}\\{x^{m}}&{0}\end{pmatrix},\,\begin{pmatrix}{0}&{0}\\{0}&{x^{m}}\end{pmatrix},\quad m=0,1,2,\ldots
\]
Specifically,
\begin{align}
\displaybreak[2]
P_{m}^{1}(x,t)&=\begin{pmatrix}{p_{2m-1}(x,t)}&{0}\\{0}&{-p_{2m}(x,t)}\end{pmatrix}, &
P_{m}^{2}(x,t)&=\begin{pmatrix}{0}&{p_{2m-1}(x,t)}\\{p_{2m}(x,t)}&{0} \end{pmatrix},\notag\\
P_{m}^{3}(x,t)&=\begin{pmatrix}{0}&{p_{2m}(x,t)}\\{p_{2m-1}(x,t)}&{0} \end{pmatrix}, &
P_{m}^{4}(x,t)&=\begin{pmatrix}{-p_{2m}(x,t)}&{0}\\{0}&{p_{2m-1}(x,t)} \end{pmatrix},\notag
\end{align}
where
\begin{equation}\label{Eq wavepolinomials}
p_{-1}\equiv 0,\quad p_{0}\equiv 1, \quad
p_{2m-1}(x,t)=\sum_{\text{even}\,k=0}^{m}\binom{m}{k}x^{m-k}t^{k}, \quad
p_{2m}(x,t)=\sum_{\text{odd}\, k=1}^{m}\binom{m}{k}x^{m-k}t^{k}.
\end{equation}

\begin{remark}\label{Rk WavePpolynomials} Each of the wave matrix not only satisfies  equation \eqref{Eq IntKernel0} but also the wave-matrix equation $\partial_{t}^{2}K(x,t)=\partial_{x}^{2}K(x,t)$. In addition, the functions in \eqref{Eq wavepolinomials} are known as wave polynomials, the fact that they arise here does not cause us any surprise since the wave polynomials form a complete system of solutions of the wave equation with respect to the maximum norm, see \cite{KKTT} for more details.
\end{remark}

Due to Proposition \ref{lm GenSolForIntKernelEq0} and the Weierstrass Approximation theorem, we establish that the wave matrices form a complete system of solutions for equation \eqref{Eq IntKernel0}.

\begin{proposition}\label{Pp CompletenessForIntKernel0} Let $k$ be a solution of equation \eqref{Eq IntKernel0}. Given any $\varepsilon>0$, there exists a linear combination of wave matrices in the form
\begin{equation*}
k_{m}(x,t)=\sum_{n=0}^{m}\Bigl(a_{n}P_{n}^{1}(x,t)+b_{n}P_{n}^{2}(x,t)+c_{n}P_{n}^{3}(x,t)+d_{n}P_{n}^{4}(x,t)\Bigr)
\end{equation*}
such that for every $(x,t)\in \left[-b,b\right]\times \left[-b,b\right]$,
\[
\left|k(x,t)-k_{m}(x,t)\right|<\varepsilon.
\]
\end{proposition}

Having disposed of this preliminary step, we are in position to introduce a complete system of solutions for equation \eqref{Eq IntKernel}. Recall that the transmutation operator $T$ in \eqref{Eq VolIntOperator} acts on vector-valued functions of one real variable, however in a natural way $T$ acts on $2\times 2$ matrix-valued functions, acting by column with respect to the variable $x$ for each fixed $t$.

\begin{definition}[Generalized Wave Matrices]\label{Df GWaveMatrices} We introduce the following matrix-valued functions being the images under the transmutation operator $T$ of the wave matrices $P_{m}^{i}$.
\begin{equation}\label{GWaveMatrices}
\begin{aligned}
& \mathcal O_{m}^{1}(x,t)=\begin{bmatrix}{\mathcal U_{2m-1}(x,t)}&{-\mathcal V_{2m}(x,t)}\end{bmatrix}, & \mathcal O_{m}^{2}(x,t)&=\begin{bmatrix}{\mathcal V_{2m}(x,t)}&{\mathcal U_{2m-1}(x,t)}\end{bmatrix},\\
& \mathcal O_{m}^{3}(x,t)=\begin{bmatrix}{\mathcal V_{2m-1}(x,t)}&{\mathcal U_{2m}(x,t)}\end{bmatrix}, & \mathcal O_{m}^{4}(x,t)&=\begin{bmatrix}{-\mathcal U_{2m}(x,t)}&{\mathcal V_{2m-1}(x,t)}\end{bmatrix},
\end{aligned}
\end{equation}
where  $m\ge 0$ and the vector-valued functions $\mathcal U_{2m-1}$, $\mathcal U_{2m}$, $\mathcal V_{2m-1}$ and $\mathcal V_{2m}$ are given by
\begin{align}
\mathcal U_{2m-1}(x,t)&=\sum_{\text{even}\, k=0}^{m}\binom{m}{k}\Phi_{m-k}(x)t^{k}, & \mathcal U_{2m}(x,t)=\sum_{\text{odd}\, k=1}^{m}\binom{m}{k}\Phi_{m-k}(x)t^{k},\label{Df GWaveMatricesUn}\\
\mathcal V_{2m-1}(x,t)&=\sum_{\text{even}\, k=0}^{m}\binom{m}{k}\Psi_{m-k}(x)t^{k}, &\mathcal V_{2m}(x,t)=\sum_{\text{odd}\, k=1}^{m}\binom{m}{k}\Psi_{m-k}(x)t^{k}.\label{Df GWaveMatricesVn}
\end{align}
\end{definition}

In what follows we are interested in linear combinations of the form
\begin{equation}\label{Eq ApproxIntKernelbyK_n}
K_{N}(x,t)=\sum_{n=0}^{N}
\Bigl(
a_{n}\mathcal O_{n}^{1}(x,t) +b_{n}\mathcal O_{n}^{2}(x,t)+c_{n}\mathcal O_{n}^{3}(x,t)+d_{n}\mathcal O_{n}^{4}(x,t)
\Bigr).
\end{equation}
Let us collect the coefficients $\left\{a_{n}, b_{n}, c_{n}, d_{n}\right\}$ in a $2\times 2$ matrix as follows
\begin{equation}\label{eq coefficientsmatrix}
C_{n}=\begin{pmatrix}{a_{n}}&{b_{n}}\\
                    {c_{n}}&{d_{n}}
			\end{pmatrix},
			\qquad
			n=0,\ldots,N.
\end{equation}
Observe that each generalized wave matrix $\mathcal{O}_{n}^{i}$, $i=1\ldots 4$, contains terms with powers of $t$ whose degree is less than or equal to $n$. Hence a linear combination of wave matrices in \eqref{Eq ApproxIntKernelbyK_n} is actually a matrix-valued polynomial function in the variable $t$. From a long but simple procedure we obtain the lemma below.

\begin{lemma}\label{Lm KnAsMatrixPolyFunction} Let $K_{N}$ be a linear combination of generalized wave matrices of the form \eqref{Eq ApproxIntKernelbyK_n}. For each fixed $x\in \left[0,b\right]$, $K_{N}$ is a polynomial in the variable $t$  whose degree is less than or equal to $N$. To be more precise,
\begin{equation}\label{Eq KnWithKnEvenKnOdd}
K_{N}(x,t)=\sum_{n=0}^N\mathcal{K}_{n}(x)t^{n},
\end{equation}
where
\begin{align}
\mathcal{K}_{2n}(x)&=\sum_{k=0}^{N-2n}\binom{2n+k}{2n}\begin{bmatrix}{\Phi_{k}(x)}&{\Psi_{k}(x)}\end{bmatrix}C_{2n+k},\label{Eq KnEvenPowert}\\
\mathcal{K}_{2n+1}(x)&=\sum_{k=0}^{N-2n-1}\binom{2n+1+k}{2n+1}\begin{bmatrix}{\Phi_{k}(x)}&{\Psi_{k}(x)}\end{bmatrix}BC_{2n+1+k}B.\label{Eq KnOddPowert}
\end{align}
\end{lemma}

Due to the transmutation property the generalized wave matrices satisfy the integral kernel equation \eqref{Eq IntKernel}. In addition,  based on the properties of operators $T$, $T^{-1}$ and Proposition \ref{Pp CompletenessForIntKernel0} we obtain that the generalized wave matrices are a complete system of solutions for  equation  \eqref{Eq IntKernel} in the square $[-b,b]\times[-b,b]$. However the integral kernel $K$ is the solution of \eqref{Eq IntKernel} only in the region $\Omega^+$, see \eqref{Omega pm} (and in $\Omega^-$ if we consider extension of the potential $Q$ onto $[-b,b]$), and is not defined in the whole square $(x,t)\in [-b,b]\times[-b,b]$. It is possible to continue $K$ as a solution of \eqref{Eq IntKernel} onto the whole square $[-b,b]\times[-b,b]$ similarly to \cite{KrT2012}, we left the details for the reader. In the present paper we give the proof for stronger result considering solutions of \eqref{Eq IntKernel} in the domain $\Omega^+$ only, without any need for continuation.

\begin{theorem}\label{Th CompletenessForIntKernelEq}
The system of matrix-valued functions $\left\{\mathcal O_{m}^{i},\, i=1,2,3,4\right\}_{m=0}^{\infty}$ is a complete system of solutions of the equation
\begin{equation}\label{Eq ThCompletenessForIntKernel}
BK_{x}(x,t)+K_{t}(x,t)B=-Q(x)K(x,t)
\end{equation}
in $\Omega^{+}$. That is, let $\tilde K$ be a solution of \eqref{Eq ThCompletenessForIntKernel}. Then for any $\varepsilon>0$ there exist a constant $N$ and coefficients $\{a_n, b_n, c_n,d_n\}_{n=0}^N$ such that
\begin{equation}\label{Completeness of O}
\sup_{(x,t)\in\Omega^+} \left|\tilde K(x,t) - \sum_{n=0}^{N}
\Bigl(
a_{n}\mathcal O_{n}^{1}(x,t) +b_{n}\mathcal O_{n}^{2}(x,t)+c_{n}\mathcal O_{n}^{3}(x,t)+d_{n}\mathcal O_{n}^{4}(x,t)
\Bigr)\right|<\varepsilon.
\end{equation}
\end{theorem}

\begin{remark}It should be noted that $\tilde K$ can be a mild solution, i.e., to be a solution of an equivalent integral equation. We only need that $\tilde K$ is continuous in the region $\Omega^+$.
\end{remark}

\begin{proof}
Consider the values of the solution $\tilde K$ at $x=b$. It is a continuous matrix valued function of the variable $t$, hence by the Weierstrass approximation theorem there exists a constant $N$ and a $2\times 2$ matrix $P_N$ whose entries are polynomials of degree less or equal to $N$ such that
\begin{equation}\label{Estimate at xb}
\sup_{t\in[-b,b]}|\tilde K(b,t) - P_N(t)|\le \frac{\varepsilon}2 \exp\biggl(-\int_0^b |Q(s)|ds\biggr).
\end{equation}
Let
\begin{equation}\label{Eq ApproxPolyDevoreLorentz}
P_N(t)=\sum_{n=0}^{N}t^{n}\begin{pmatrix}{\tilde{a}_{n}}&{\tilde{b}_{n}}\\{\tilde{c}_{n}}&{\tilde{d}_{n}}\end{pmatrix}.
\end{equation}

Now we consider the following Cauchy problem for equation \eqref{Eq IntKernel} in the region $\Omega^+$
\begin{equation}\label{Pb CauchyPblemEqForK}
   \left\{
	\begin{aligned}
        BK_{x}(x,t)+K_{t}(x,t)B &=-Q(x)K(x,t),\\
        K(b,t)&= F(t),
    \end{aligned}
   \right.
\end{equation}
where $F\in C([-b,b],\mathcal{M}_2)$.
As it is shown in Appendix \ref{Append4}, this Cauchy problem is well-posed and its solution satisfies
\begin{equation}\label{EstimateCauchyProblem}
    \sup_{(x,t)\in\Omega^+} |K(x,t)| < 2\sup_{s\in [-b,b]}|F(s)| \cdot e^{b \int_0^b |Q(\tau)|d\tau}.
\end{equation}

Note that $\tilde K$ is the solution of Cauchy problem \eqref{Pb CauchyPblemEqForK} with $F(t)=\tilde K(b,t)$. Let $K_2$ denote the solution of Cauchy problem \eqref{Pb CauchyPblemEqForK} with $F=P_N$. The difference $\tilde K-K_2$ is the solution of Cauchy problem \eqref{Pb CauchyPblemEqForK} with $F(t) = \tilde K(b,t) - P_N(t)$. Then it follows from \eqref{Estimate at xb} and \eqref{EstimateCauchyProblem} that
\[
\sup_{(x,t)\in\Omega^+} |\tilde K(x,t)-K_2(x,t)|<\varepsilon,
\]
hence it is sufficient to show that $K_2$ is a linear combination of the generalized wave polynomials to finish the proof. For that we show that the equation
\begin{equation}\label{Eq for p}
P_N(t)=\sum_{n=0}^{N}
\Bigl(
a_{n}\mathcal O_{n}^{1}(b,t) +b_{n}\mathcal O_{n}^{2}(b,t)+c_{n}\mathcal O_{n}^{3}(b,t)+d_{n}\mathcal O_{n}^{4}(b,t)
\Bigr)
\end{equation}
possesses a solution $\{a_n, b_n, c_n,d_n\}_{n=0}^{N}$. Due to Lemma \ref{Lm KnAsMatrixPolyFunction}
\begin{equation*}
\sum_{n=0}^{N}
\Bigl(
a_{n}\mathcal O_{n}^{1}(b,t) +b_{n}\mathcal O_{n}^{2}(b,t)+c_{n}\mathcal O_{n}^{3}(b,t)+d_{n}\mathcal O_{n}^{4}(b,t)
\Bigr)=\sum_{n=0}^{N}\mathcal{K}_{n}t^{n},
\end{equation*}
where
\begin{align}
\displaybreak[2]
\mathcal{K}_{2n}&=\sum_{k=0}^{N-2n}\binom{2n+k}{2n}\begin{bmatrix}{\Phi_{k}(b)}&{\Psi_{k}(b)}\end{bmatrix}C_{2n+k},\label{Eq KnEvenPowert1}\\
\mathcal{K}_{2n+1}&=\sum_{k=0}^{N-2n-1}\binom{2n+1+k}{2n+1}\begin{bmatrix}{\Phi_{k}(b)}&{\Psi_{k}(b)}\end{bmatrix}BC_{2n+1+k}B.\label{Eq KnOddPowert1}
\end{align}
Suppose that $N$ is even. By equating the coefficients at $t^N$ from  \eqref{Eq KnEvenPowert1} and \eqref{Eq ApproxPolyDevoreLorentz} we obtain the following equation
\begin{equation}\label{Eq1}
\begin{pmatrix}{\tilde{a}_{N}}&{\tilde{b}_{N}}\\
               {\tilde{c}_{N}}&{\tilde{d}_{N}}
\end{pmatrix}=
\begin{bmatrix}
{\Phi_{0}(b)}&{\Psi_{0}(b)}
\end{bmatrix}
\begin{pmatrix}{a_{N}}&{b_{N}}\\
               {c_{N}}&{d_{N}}
\end{pmatrix}.
\end{equation}
Since the vectors  $\Phi_{0}(b)$, $\Psi_{0}(b)$ are linearly independent, the $2\times 2$ matrix $\begin{bmatrix}{\Phi_{0}(b)}&{\Psi_{0}(b)}\end{bmatrix}$ is invertible. Hence, in a unique way, we determine the coefficients $a_N$, $b_N$, $c_N$ and $d_N$ from \eqref{Eq1}. For an odd $N$ the procedure is similar, the only difference is that equation \eqref{Eq KnOddPowert1} is used.
Now we proceed by induction: by subtracting the terms corresponding to $N$ from \eqref{Eq for p}, we obtain the problem containing powers of $t$ of degree at most $N-1$ and similar reasoning works.

Consider the function
\begin{equation}\label{Eq3}
\sum_{n=0}^{N}
\Bigl(
a_{n}\mathcal O_{n}^{1}(x,t) +b_{n}\mathcal O_{n}^{2}(x,t)+c_{n}\mathcal O_{n}^{3}(x,t)+d_{n}\mathcal O_{n}^{4}(x,t)
\Bigr).
\end{equation}
As a linear combination of generalized wave polynomials, it is a solution of \eqref{Eq IntKernel}, and by construction satisfies \eqref{Eq for p}. Hence \eqref{Eq3} is the solution of the Cauchy problem \eqref{Pb CauchyPblemEqForK} and coincides with the function $K_2$.
\end{proof}

\section{Approximation of the integral kernel $K$}\label{Sect5}

Theorem \ref{Th CompletenessForIntKernelEq} guarantees the existence of coefficients $\left\{a_{n}, b_{n}, c_{n}, d_{n} \right\}_{n=0}^{N}\subseteq \mathbb C$ and a linear combination in the form \eqref{Eq ApproxIntKernelbyK_n}
approximating the integral kernel $K$.

To convert this existence result into a practical scheme for obtaining the coefficients  $\left\{a_{n}, b_{n}, c_{n}, d_{n} \right\}_{n=0}^{N}$ for any given potential matrix $Q$, we are going to utilize the Goursat conditions \eqref{Eq GsatDataX+}, \eqref{Eq GsatDataX-}. Restricting the inequality \eqref{Completeness of O} to the characteristic curves $t=x$ and $t=-x$ one can see that arbitrary close approximation of the Goursat data is always possible.
In Appendix \ref{Append1} we show that having sufficiently good approximation of the Goursat data on the characteristics curves $t=x$ and $t=-x$, a good approximation of the integral kernel $K$ in form \eqref{Eq ApproxIntKernelbyK_n} is guaranteed on the whole $\Omega^{+}$.

In this section we show that by considering the half-sum and half-difference of the Goursat conditions, the problem of obtaining coefficients $\left\{a_{n}, b_{n}, c_{n}, d_{n} \right\}_{n=0}^{N}$ from two conditions \eqref{Eq GsatDataX+}, \eqref{Eq GsatDataX-} can be reduced to the problem of obtaining the coefficients $\left\{a_{n}, b_{n}, c_{n}, d_{n} \right\}_{n=0}^{N}$ from only one condition, for which the least squares method can be applied.

First of all, note that $B^2=-I$, hence for any matrix-valued function $K(x,t)$ one has
\begin{align}
    BK(x,x) - K(x,x)B &= BK(x,x) + BBK(x,x)B = 2B\mathcal{P}^+[K(x,x)],\label{FirstConditionP}\\
    BK(x,-x) + K(x,-x)B &= BK(x,-x) - BBK(x,-x)B = 2B\mathcal{P}^-[K(x,-x)].\label{SecondConditionP}
\end{align}
Using the definitions \eqref{GWaveMatrices}, \eqref{Df GWaveMatricesUn}, \eqref{Df GWaveMatricesVn} and the parity properties,
one can obtain the following result.
\begin{lemma}\label{lm KN_ValuesOnGsatD} Under the same notation as in \eqref{eq coefficientsmatrix}, one has
\begin{align}
\displaybreak[2]
\mathcal{P}^{+}\left[K_{N}(x,x)\right]&=
\sum_{n=0}^{N}\left(
\mathcal{M}_{n}(x)\frac{C_{n}}{2}+\mathcal {N}_{n}(x)\frac{C_{n}}{2}B\right),\label{FirstConditionPK}\\
\mathcal{P}^{-}\left[K_{N}(x,-x)\right]&=\sum_{n=0}^{N}
\left(\mathcal {M}_{n}(x)\frac{C_{n}}{2}-\mathcal {N}_{n}(x)\frac{C_{n}}{2}B\right),\label{SecondConditionPK}
\end{align}
where the matrix valued functions $\mathcal {N}_{n}$ and $\mathcal {M}_{n}$ are given by
\begin{align}
\mathcal {M}_{n}(x)&=
\begin{bmatrix}
\mathcal U_{2n-1}(x,x)+B\mathcal V_{2n}(x,x) & \mathcal V_{2n-1}(x,x)-B\mathcal U_{2n}(x,x)
\end{bmatrix}\label{Eq halfDifferenceKn},\\
\mathcal {N}_{n}(x)&=
\begin{bmatrix}
-\mathcal V_{2n}(x,x)+B\mathcal U_{2n-1}(x,x) & \mathcal U_{2n}(x,x)+B\mathcal V_{2n-1}(x,x)
\end{bmatrix}.\label{Eq halfSumKn}
\end{align}
Moreover, the following relation holds:
\begin{equation}\label{Eq RelationNnandMn}
B\mathcal{N}_{n}(x)=-\mathcal{M}_{n}(x).
\end{equation}
\end{lemma}

\begin{proof}
Equality \eqref{Eq RelationNnandMn} follows directly from \eqref{Eq halfDifferenceKn} and \eqref{Eq halfSumKn} by using $B^2=-I$.

Let us verify \eqref{FirstConditionPK}, the second equality is similar. Note that
\[
C_nB = \begin{pmatrix} a_n & b_n \\ c_n & d_n \end{pmatrix} \begin{pmatrix} 0 & 1 \\ -1 & 0 \end{pmatrix} =
\begin{pmatrix} -b_n & a_n \\ -d_n & c_n \end{pmatrix}.
\]
Thus
\begin{align*}
    \mathcal{M}_n C_n &= \begin{bmatrix}
    a_n(\mathcal{U}_{2n-1}+B\mathcal{V}_{2n})+c_n(\mathcal{V}_{2n-1}-B\mathcal{U}_{2n}) &
    b_n(\mathcal{U}_{2n-1}+B\mathcal{V}_{2n})+d_n(\mathcal{V}_{2n-1}-B\mathcal{U}_{2n})
    \end{bmatrix},\\
    \mathcal{N}_n C_nB &= \begin{bmatrix}
    -b_n(-\mathcal{V}_{2n}+B\mathcal{U}_{2n-1})-d_n(\mathcal{U}_{2n}-B\mathcal{V}_{2n-1}) &
    a_n(-\mathcal{V}_{2n}+B\mathcal{U}_{2n-1})+c_n(\mathcal{U}_{2n}-B\mathcal{V}_{2n-1})
    \end{bmatrix}
\end{align*}
and finally
\begin{align*}
\displaybreak[2]
\mathcal{M}_nC_n + \mathcal{N}_n C_nB &= a_n\begin{bmatrix}
\mathcal{U}_{2n-1}+B\mathcal{V}_{2n} & -\mathcal{V}_{2n}+B\mathcal{U}_{2n-1}
\end{bmatrix} + b_n\begin{bmatrix}
\mathcal{V}_{2n}-B\mathcal{U}_{2n-1} & \mathcal{U}_{2n-1}+B\mathcal{V}_{2n}
\end{bmatrix}\\
&\quad + c_n\begin{bmatrix}
\mathcal{V}_{2n-1}-B\mathcal{U}_{2n} & \mathcal{U}_{2n}-B\mathcal{V}_{2n-1}
\end{bmatrix} +d_n\begin{bmatrix}
-\mathcal{U}_{2n}+B\mathcal{V}_{2n-1} & \mathcal{V}_{2n-1}-B\mathcal{U}_{2n}
\end{bmatrix}.
\end{align*}
Note that
\begin{align*}
a_n&\begin{bmatrix}
\mathcal{U}_{2n-1}+B\mathcal{V}_{2n} & -\mathcal{V}_{2n}+B\mathcal{U}_{2n-1}
\end{bmatrix}\\
& = a_n\begin{bmatrix}
\mathcal{U}_{2n-1} & -\mathcal{V}_{2n}
\end{bmatrix}
+ a_nB\begin{bmatrix}
\mathcal{V}_{2n} & \mathcal{U}_{2n-1}
\end{bmatrix}\\
&= a_n \mathcal{O}_n^1 + a_n B \begin{bmatrix}
 \mathcal{U}_{2n-1} & -\mathcal{V}_{2n}
\end{bmatrix}B
 = a_n \mathcal{O}_n^1 + a_n B \mathcal{O}_n^1 B,
\end{align*}
similarly for the other terms. Hence
\[
\mathcal{M}_nC_n + \mathcal{N}_n C_nB  = a_n \mathcal{O}_n^1 + b_n \mathcal{O}_n^2 +c_n \mathcal{O}_n^3 + d_n \mathcal{O}_n^4
 + B\bigl(a_n \mathcal{O}_n^1 + b_n \mathcal{O}_n^2 + c_n \mathcal{O}_n^3 + d_n \mathcal{O}_n^4\bigr) B,
\]
exactly twice the expression in the left-hand side of \eqref{FirstConditionPK}.
\end{proof}

Expressions \eqref{FirstConditionPK} and \eqref{SecondConditionPK} suggest to consider half-sum and half-difference.

\begin{lemma}\label{Lm halfSumDifferenceKn} The half-sum and half-difference of the Goursat data corresponding to $K_{N}$ have the following form:
\begin{align}
B\mathcal{P}^{+}\left[K_{N}(x,x)\right]+B\mathcal{P}^{-}\left[K_{N}(x,-x)\right]&=\sum_{n=0}^{N}\mathcal {N}_{n}(x)C_n,\label{Eq4}
\\
B\mathcal{P}^{+}\left[K_{N}(x,x)\right]-B\mathcal{P}^{-}\left[K_{N}(x,-x)\right]&=\sum_{n=0}^{N}B\mathcal {N}_{n}(x)C_nB,\label{Eq5}
\end{align}
\end{lemma}

The principal significance of this lemma is that the two conditions \eqref{Eq4} and \eqref{Eq5} are equivalent. Indeed, using \eqref{FirstConditionP} and \eqref{SecondConditionP} one easily verifies that
\[
B\left[B\mathcal{P}^{+}\left[K_{N}(x,x)\right]+B\mathcal{P}^{-}\left[K_{N}(x,-x)\right]\right]B = B\mathcal{P}^{+}\left[K_{N}(x,x)\right]-B\mathcal{P}^{-}\left[K_{N}(x,-x)\right],
\]
showing that the condition \eqref{Eq5} is nothing more than \eqref{Eq4} multiplied by the matrix $B$ from both sides. As we show below, this fact allows one to obtain the coefficients $\left\{a_{n}, b_{n}, c_{n}, d_{n} \right\}_{n=0}^{N}$ from one condition only.

As was mentioned at the beginning of this section, arbitrary close uniform approximation of the Goursat conditions \eqref{Eq GsatDataX+}, \eqref{Eq GsatDataX-} by a linear combination of the form \eqref{Eq ApproxIntKernelbyK_n} is always possible. Hence it follows from \eqref{Eq4} that an arbitrary close approximation of $-Q(x)/2$ by an expression of the form $\sum_{n=0}^{N}\mathcal {N}_{n}(x)C_n$ is possible. So we may formulate the following result.

\begin{lemma}\label{Lm ApproxQ} Suppose that  $\left\{a_{n}, b_n, c_{n}, d_n\right\}_{n=0}^{N}$ are complex numbers such that
\begin{equation}\label{eq UniformProblem}
\left|
-\frac{1}{2}Q(x)
-\sum_{n=0}^{N}\mathcal {N}_{n}(x)\begin{pmatrix}a_{n} & b_n\\c_{n} & d_n
\end{pmatrix}
\right|
<\varepsilon
\end{equation}
and let $K_N$ be defined by \eqref{Eq ApproxIntKernelbyK_n}.
Then
\begin{equation*}
\left|
 -Q(x)-\bigl(BK_N(x,x) - K_N(x,x)B\bigr)
\right|
<2\varepsilon
\qquad\text{and}\qquad
\left|
BK_N(x,-x) + K_N(x,-x)B
\right|
<2\varepsilon.
\end{equation*}
\end{lemma}

\begin{proof} Taking into account that $BQ(x)B=Q(x)$ we obtain from \eqref{eq UniformProblem} that
\begin{equation}\label{eq UniformProblem2}
\left|
\frac{1}{2}Q(x)
-\sum_{n=0}^{N}B\mathcal {N}_{n}(x)\begin{pmatrix}a_{n} & b_n\\c_{n} & d_n
\end{pmatrix}B
\right|
<\varepsilon.
\end{equation}
Using \eqref{Eq4} and \eqref{Eq5} we obtain, taking half-sum and half-difference of the expressions in \eqref{eq UniformProblem} and \eqref{eq UniformProblem2}, that
\begin{gather*}
\left|Q(x) - 2B\mathcal{P}^+ [K_N(x,x)]\right|  < 2\varepsilon,\\
\left|B\mathcal{P}^{-}\left[K_{N}(x,-x)\right]\right|  < 2\varepsilon,
\end{gather*}
exactly the expressions from the statement of Lemma, c.f., \eqref{FirstConditionP}, \eqref{SecondConditionP}.
\end{proof}

\begin{theorem}\label{Th AATOForDiracUniform}
Let $Q\in C([0,b],\mathcal{M}_2)$. Suppose that  $\left\{a_{n}, b_n, c_{n}, d_n\right\}_{n=0}^{N}$ be complex numbers such that for every $x\in \left[0,b \right],$
\begin{equation*}
\left|
-\frac{1}{2}Q(x)
-\sum_{n=0}^{N}\mathcal {N}_{n}(x)\begin{pmatrix}a_{n} & b_n\\c_{n} & d_n
\end{pmatrix}
\right|
<\varepsilon
\end{equation*}
Then the integral kernel $K$ is approximated by the linear combination \eqref{Eq ApproxIntKernelbyK_n}
in such a way that the following inequality holds
\begin{equation}\label{Eq EstimateIntegralKernerUniform}
\sup_{(x,t)\in\Omega^+}\left|K(x,t)-K_{N}(x,t)\right|<C_{Q}\varepsilon,
\end{equation}
here the constant $C_{Q}$ depends on the potential $Q$, but does not depend on $N$.
\end{theorem}

\begin{proof}
The proof immediately follows from Lemma \ref{Lm ApproxQ} and well-posedness of the Goursat problem \eqref{Eq IntKernel}--\eqref{Eq GsatDataX-}, see Appendix \ref{Append1}.
\end{proof}

\section{Least squares approximation and convergence rate estimate}\label{Sect6}

Despite Lemma \ref{Lm ApproxQ} shows how to find the coefficients of the approximation $K_N$, there exists a disadvantage from the practical point of view since minimizing \eqref{eq UniformProblem} in a uniform norm is not a simple task, even more in the context of matrix-valued functions, see \cite{DevoreLorentz}. While the least squares method can be used to minimize \eqref{eq UniformProblem}, it does not need to produce nice uniform approximation (Gibbs phenomenon can occur, for example).
In this section we present error estimates for the approximate integral kernel obtained using the least squares minimization. Additionally, we prove decay rate estimates depending on the smoothness of the potential matrix $Q$.

Consider the Goursat problem
\begin{align}
BK_{x}(x,t)+K_{t}(x,t)B &=-Q(x)K(x,t),\label{Eq DifferenceOfTwoSolts}\\
BK(x,x)-K(x,x)B &=E_{1}(x),\label{Eq GsatDataforDifference+}\\
BK(x,-x)+K(x,-x)B&=E_{2}(x),\label{Eq GsatDataforDifference-}
\end{align}
where $E_{1}$ and $E_{2}$ satisfy the compatibility conditions $E_{1}\in\mathcal{H}^{-}$ and $E_{2}\in\mathcal{H}^{+}$. Here $\mathcal{H}=L_2((0,b), \mathcal{M}_2)$ and the subspaces $\mathcal{H}^{\pm}$ are defined as in Section \ref{Sect4}.

Any difference $\left[K-K_{N}\right](x,t)$ satisfies a problem having the form \eqref{Eq DifferenceOfTwoSolts}--\eqref{Eq GsatDataforDifference-}. The least squares method minimizes the $L_2((0,b), \mathcal{M}_2)$ norm of the difference in \eqref{eq UniformProblem}, and by Lemma \ref{Lm ApproxQ}, $L_2((0,b), \mathcal{M}_2)$ norms of the functions $E_{1,2}$ are bounded by twice the norm of \eqref{eq UniformProblem}. Naturally we are interested in estimating the solution of the Goursat promlem \eqref{Eq DifferenceOfTwoSolts}--\eqref{Eq GsatDataforDifference-} in terms of the norms of the functions $E_{1,2}$. The following proposition holds.

\begin{proposition}\label{Pp EstimateInL2x-Norm}
Let $E_{1}$, $E_{2}$, $Q\in L_{2}\bigl((0,b),\mathcal{M}_2\bigr)$. Let $E_{1}$ and $E_{2}$ satisfy the compatibility conditions $E_{1}\in\mathcal{H}^{-}$ and $E_{2}\in\mathcal{H}^{+}$. Then the Goursat problem \eqref{Eq DifferenceOfTwoSolts}--\eqref{Eq GsatDataforDifference-} possesses a solution $K$ such that for all $x\in \left[0, b\right]$ the following estimate holds
\begin{equation}\label{EstimateInL2}
\int_{-x}^{x}\left|K(x,t)\right|^{2}\,dt \le \frac 12\left(\left\|E_{1}\right\|_{L_2(0,b)}^{2}+\left\|E_{2}\right\|_{L_2(0,b)}^{2}\right) \left(1+\sqrt b \|Q\|_{L_2(0,b)} \cdot e^{\sqrt b\|Q\|_{L_2(0,b)}}\right).
\end{equation}
\end{proposition}

\begin{proof}For the proof of this proposition we refer the reader to Appendix \ref{Append1}.\end{proof}

Now, suppose that potential matrix $Q$ is a smooth function. Let $Q\in C^r\bigl([0,b],\mathcal{M}_2\bigr)$ for some $r\in \mathbb{N}_0 = \mathbb{N}\cup\{0\}$.
Then the statement of Theorem \ref{Th CompletenessForIntKernelEq} can be made more precise.

\begin{proposition}\label{Prop completeness smooth Q}
Let $Q\in C^r\bigl([0,b],\mathcal{M}_2\bigr)$ for some $r\in \mathbb{N}_0$ and $K$ be the integral kernel of transmutation operator. Then there exists a constant $C>0$ such that for every $N>r$ there exist coefficients $\{a_n, b_n, c_n,d_n\}_{n=0}^N$ such that
\begin{equation}\label{Completeness of O sm Q}
\sup_{(x,t)\in\Omega^+} \left|K(x,t) - \sum_{n=0}^{N}
\Bigl(
a_{n}\mathcal O_{n}^{1}(x,t) +b_{n}\mathcal O_{n}^{2}(x,t)+c_{n}\mathcal O_{n}^{3}(x,t)+d_{n}\mathcal O_{n}^{4}(x,t)
\Bigr)\right|<\frac{C}{N^r}.
\end{equation}
\end{proposition}

\begin{proof}
Under the condition $Q\in C^r\bigl([0,b],\mathcal{M}_2\bigr)$ the integral kernel $K\in C^r\bigl(\Omega^+, \mathcal{M}_2\bigr)$, see Appendix \ref{Append2}. Hence $K(b,\cdot)\in C^r\bigl([-b,b],\mathcal{M}_2\bigr)$. Consider the polynomial-matrix $P_N$ of best uniform approximation of the function $K(b,\cdot)$. As it follows from \cite[Chapter 7, Theorem 6.2]{DevoreLorentz},
\[
\sup_{t\in[-b,b]}\left|K(b,t) - P_N(t)\right| <\frac{C}{N^r}
\]
for every $N>r$, where the constant $C$ does not depend on $N$. Now the proof can be finished identically to the proof of Theorem \ref{Th CompletenessForIntKernelEq}.
\end{proof}

\begin{remark}
Actually, the right-hand side in \eqref{Completeness of O sm Q} can be changed to $o\left(\frac 1{N^r}\right)$, $N\to\infty$.
\end{remark}

\begin{corollary}\label{Cor K SmoothQ}
Let $Q\in C^r\bigl([0,b],\mathcal{M}_2\bigr)$ for some $r\in \mathbb{N}_0$ and $K$ be the integral kernel of the transmutation operator. Then for every $N>r$  there exists an approximate kernel $K_N$ of the form \eqref{Eq ApproxIntKernelbyK_n} such that
\begin{gather}
    \sup_{(x,t)\in Q^+}\bigl|K(x,t) - K_N(x,t)\bigr| \le \frac{C}{N^r}, \label{Estimate K Uniform}\\
    \sup_{x\in [0,b]} \bigl| -Q - (BK_N(x,x) - K_N(x,x)B)\bigr| < \frac{2C}{N^r}, \label{Estimate GC1}\\
    \sup_{x\in [0,b]} \bigl| BK_N(x,-x) + K_N(x,-x)B)\bigr| < \frac{2C}{N^r}, \label{Estimate GC2}
\end{gather}
where the constant $C$ does not depend on $N$.
\end{corollary}

\begin{proof}
Inequality \eqref{Estimate K Uniform} immediately follows from Proposition \ref{Prop completeness smooth Q}. To obtain \eqref{Estimate GC1}, we apply the triangle inequality
\[
\begin{split}
\bigl|-Q - (BK_N - K_NB)\bigr| &= \bigl|-Q - (BK - KB) - (B(K_N - K) - (K_N-K)B)\bigr|\\
&\le \bigl|-Q - (BK - KB)\bigr| + \bigl|B(K_N - K)\bigr| + \bigl|(K_N-K)B)\bigr| < \frac{2C}{N^r},
\end{split}
\]
where we have used that $K$ satisfies \eqref{Eq GsatDataforDifference+} with $E_1=-Q$ and that multiplication by the matrix $B$ does not change the matrix norm. Similarly for \eqref{Estimate GC2}.
\end{proof}

Consider the following minimization problem
\begin{equation}\label{LeastSquaresMin}
\left\{a_{n}, b_n, c_{n}, d_n\right\}_{n=0}^{N} = \operatorname*{arg\, min}_{\left\{a_{n}, b_n, c_{n}, d_n\right\}_{n=0}^{N}\subset \mathbb{C}}
\left\|
\frac{1}{2}Q(x)
+\sum_{n=0}^{N}\mathcal {N}_{n}(x)\begin{pmatrix}a_{n} & b_n\\c_{n} & d_n
\end{pmatrix}
\right\|^2_{L_2(0,b)}.
\end{equation}
We refer the reader to Appendix \ref{Append5} for reduction of this minimization problem to the solution of two systems of linear equations.

\begin{theorem}\label{Th AATOForDiracL2norm}Let $Q\in C^r\bigl([0,b],\mathcal{M}_2\bigr)$ for some $r\in \mathbb{N}_0$ and $K$ be the integral kernel of the transmutation operator. For every $N>r$, let the coefficients $\left\{a_{n}, b_n, c_{n}, d_n\right\}_{n=0}^{N}\subset \mathbb{C}$ be obtained as the least squares solution of \eqref{LeastSquaresMin} and let us define the approximate kernel $K_N$ by \eqref{Eq ApproxIntKernelbyK_n}. Then for every $x\in [0,b]$ the following estimate holds
\begin{equation}\label{Eq EstimateIntegralKernerL2norm}
\left\|K(x,\cdot_{t})-K_{N}(x,\cdot_{t})\right\|_{L_{2}(-x,x)}< \frac{2C}{N^r},
\end{equation}
where the constant $C$ does not depend on $x$ and $N$.
\end{theorem}

\begin{proof}
Let $N>r$ be fixed. By Corollary \ref{Cor K SmoothQ} there exist a constant $C_1$, independent of $N$, and coefficients $\{\tilde a_{n}, \tilde b_n, \tilde c_n, \tilde d_n\}_{n=0}^{N}\subset \mathbb{C}$ such that the approximate kernel $\tilde K_N$ defined from these coefficients via \eqref{Eq ApproxIntKernelbyK_n} satisfies
\begin{gather*}
    \sup_{x\in [0,b]} \bigl| -Q - (B\tilde K_N(x,x) - \tilde K_N(x,x)B)\bigr| < \frac{2C}{N^r},\\
    \sup_{x\in [0,b]} \bigl| B \tilde K_N(x,-x) + \tilde K_N(x,-x)B)\bigr| < \frac{2C}{N^r}.
\end{gather*}
Hence the half-sum of the Goursat data satisfies (see \eqref{Eq4})
\begin{equation}\label{Estimate1}
\left|
-\frac{1}{2}Q(x)
-\sum_{n=0}^{N}\mathcal {N}_{n}(x)\begin{pmatrix}\tilde a_{n} & \tilde b_n\\ \tilde c_{n} & \tilde d_n
\end{pmatrix}
\right|< \frac{2C}{N^r}
\end{equation}
for every $x\in [0,b]$. Integrating \eqref{Estimate1} from 0 to $x$ we obtain that
\begin{equation*}
\left\|
-\frac{1}{2}Q(t)
-\sum_{n=0}^{N}\mathcal {N}_{n}(t)\begin{pmatrix}\tilde a_{n} & \tilde b_n\\ \tilde c_{n} & \tilde d_n
\end{pmatrix}
\right\|_{L_2(0,b)}< \frac{2C\sqrt b}{N^r},
\end{equation*}
that is, we have a set of coefficients guaranteing at least $\frac{2C\sqrt b}{N^r}$ as the result of  the minimization problem \eqref{LeastSquaresMin}. The least squares method provides
coefficients $\left\{a_{n}, b_n, c_{n}, d_n\right\}_{n=0}^{N}\subset \mathbb{C}$ for which the right-hand side in the problem \eqref{LeastSquaresMin} can not be larger, and the statement follows from Proposition \ref{Pp EstimateInL2x-Norm} and Lemma \ref{Lm ApproxQ}.
\end{proof}

\section{Approximation of the solutions of Dirac equation}\label{Sect7}
Let $C(x,\lambda)$ and $S(x,\lambda)$ be the solutions of the Dirac equation \eqref{DiracSystem} satisfying the initial conditions $\left(1,0\right)^{T}$ and $\left(0,1\right)^{T}$ at $x=0$ respectively. It follows from the definition of the transmutation operator that
\begin{equation*}
C(x,\lambda)  = T\begin{pmatrix}{\cos\lambda x}\\
                     {\sin\lambda x}
			\end{pmatrix}\qquad \text{and}\qquad
S(x,\lambda)=
T\begin{pmatrix}{-\sin\lambda x}\\
                     {\cos\lambda x}
			\end{pmatrix}.
\end{equation*}
Consider the following approximations to $C(x,\lambda)$ and $S(x,\lambda)$
\begin{align*}
C_{N}(x,\lambda)&=
\begin{pmatrix}{\cos\lambda x}\\
                {\sin\lambda x}
\end{pmatrix}
+
\int_{-x}^{x}
K_{N}(x,t)\begin{pmatrix}{\cos\lambda t}\\
                         {\sin\lambda t}
          \end{pmatrix}\,
dt,\\
S_{N}(x,\lambda)&=
\begin{pmatrix}{-\sin\lambda x}\\
               {\cos\lambda x}
\end{pmatrix}
+
\int_{-x}^{x}
K_{N}(x,t)\begin{pmatrix}{-\sin\lambda t}\\
                         {\cos\lambda t}
					\end{pmatrix}\,
dt,
\end{align*}
where $K_{N}$ is an approximation of the form \eqref{Eq ApproxIntKernelbyK_n} of the integration kernel $K$. According to Lemma \ref{Lm KnAsMatrixPolyFunction}, the function $K_N$ has the following form
\[
K_{N}(x,t)=\sum_{n=0}^{N}
\mathcal{K}_{n}(x)t^{n},
\]
where the coefficients  $\mathcal{K}_{n}$ are given by \eqref{Eq KnEvenPowert} and \eqref{Eq KnOddPowert}. Hence
\begin{align}
C_{N}(x,\lambda)&=
\begin{pmatrix}{\cos\lambda x}\\
                {\sin\lambda x}
\end{pmatrix}
+
\sum_{n=0}^N \mathcal{K}_{n}(x)\int_{-x}^{x}
\begin{pmatrix}{t^n\cos\lambda t}\\
                         {t^n\sin\lambda t}
          \end{pmatrix}\,
dt, \label{Eq ApproximateCosineSol_DiracSys1} \\
S_{N}(x,\lambda)&=
\begin{pmatrix}{-\sin\lambda x}\\
               {\cos\lambda x}
\end{pmatrix}
+
\sum_{n=0}^N \mathcal{K}_{n}(x)\int_{-x}^{x}
\begin{pmatrix}{-t^n\sin\lambda t}\\
                         {t^n\cos\lambda t}
					\end{pmatrix}\,
dt.\label{Eq ApproximateSineSol_DiracSys1}
\end{align}
The integrals here can be easily calculated explicitly. For example, the following formulas can be used \cite[ 2.633]{GradshteinAndRizhik}
\begin{align*}
\int t^{k} \sin\lambda t\,dt&=-\sum_{j=0}^{k}j!\binom{k}{j}\frac{t^{k-j}}{\lambda^{j+1}}\cos\left(\lambda t+\frac{j\pi}{2}\right),\\
\int t^{k} \cos\lambda t\,dt&=\sum_{j=0}^{k}j!\binom{k}{j}\frac{t^{k-j}}{\lambda^{j+1}}\sin\left(  \lambda t+\frac{j\pi}{2}\right),
\end{align*}
or alternatively the integrals can be calculated recursively.

The errors of approximations of the solutions can be bounded independently on the size of the spectral parameter, see \eqref{Estimate for Solution}. More precisely, the following result holds.
\begin{proposition}
Let $Q\in C^r\bigl([0,b],\mathcal{M}_2\bigr)$ for some $r\in \mathbb{N}_0$. For every $N>r$, let the coefficients $\left\{a_{n}, b_n, c_{n}, d_n\right\}_{n=0}^{N}\subset \mathbb{C}$ be obtained as the least squares solution of \eqref{LeastSquaresMin} and let us define approximate solutions by \eqref{Eq ApproximateCosineSol_DiracSys1} and \eqref{Eq ApproximateSineSol_DiracSys1}. Then for every $\lambda\in\mathbb{R}$ the following estimates hold
\begin{equation}\label{EstimateSols}
    \bigl|C(x,\lambda) - C_N(x,\lambda)\bigr|\le \frac{2C\sqrt{x}}{N^r}\qquad \text{and}\qquad \bigl|S(x,\lambda) - S_N(x,\lambda)\bigr|\le \frac{2C\sqrt{x}}{N^r},
\end{equation}
where the constant $C$ does not depend on $N$, $\lambda$ and $x$.
\end{proposition}

\begin{proof}
The proof follows from Theorem \ref{Th AATOForDiracL2norm} and \eqref{Estimate for Solution}.
\end{proof}

\section{Application to one-dimensional Schr\"{o}dinger equation}\label{Sect8}
Consider the equation
\begin{equation}\label{Eq Schrod}
    -y'' + q_1(x) y = \omega^2 y, \qquad x\in [0,b],
\end{equation}
where $q_1\in C[0,b]$ is a complex valued function. Let $f$ be a solution of \eqref{Eq Schrod} corresponding to $\omega=0$, i.e.,
\[
-f''+q_1f = 0,
\]
and such that $f(0) = 1$ and $f$ does not vanish on the whole $[0,b]$. Such solution always exists, see \cite[Remark 5]{KrPorter2010} and \cite{CamporesiScala2011}. Consider new functions $u = y$ and $v = \frac{1}{\omega} f\left(\frac y f\right)'$. Then (see \cite[Example 2.1]{KTG1}) equation \eqref{Eq Schrod} is equivalent to the following Dirac equation
\begin{equation}\label{Equiv Dirac}
    v' + q(x) v = \omega u,\qquad  -u'+ q(x) u = \omega v,
\end{equation}
where $q:= \frac{f'}{f}$. The system \eqref{Equiv Dirac} is also known as one-dimensional Dirac equation
with a Lorentz scalar potential, see \cite{KrT2013} and references therein.

In this section we compare the results obtained for equation \eqref{Eq Schrod} in \cite{KrT2015} with those obtained using the Dirac equation approach. Even though the integral kernel of the transmutation operator for the Dirac equation \eqref{Equiv Dirac} is formed by the integral kernels of the transmutation operators for \eqref{Eq Schrod} and its Darboux transformed equation, see \cite{KrT2012}, the proposed approach leads to a different analytic approximations.

Let us first briefly summarize some facts from \cite{KrT2015}. Let $h:=f'(0)$. There exists a transmutation operator $T_f$ for the pair of operators $-\frac{d^2}{dx^2}+q_1(x)$ and $-\frac{d^2}{dx^2}$ given by
\begin{equation}\label{Eq VolIntOperatorForSchodinger}
T_{f}u(x)=u(x)+\int_{-x}^{x}K_{f}(x,t)u(t)\,dt.
\end{equation}
Its integral kernel $K_f$ satisfies the following Goursat problem
\begin{gather*}
\Biggl(\frac{d^{2}}{dx^{2}}-q_1(x)\Biggr)K_f(x,t)=\frac{d^{2}}{dx^{2}}K_f(x,t),\\
K_f(x,x)=\frac{h}{2}+\frac{1}{2}\int_{0}^{x}q(s)\,ds, \qquad K_f(x,-x)=\frac{h}{2}.
\end{gather*}

Let us consider together with equation \eqref{Eq Schrod}  its Darboux transformed equation
\begin{equation}\label{Eq SchrodDarboux}
-\frac{d^{2}y}{dx^{2}}+q_{2}(x)y = \omega^2 y,
\end{equation}
where $q_{2}(x)=2\left(f'/f\right)^{2}-q_{1}(x)$. Note that $1/f$ is a particular solution of \eqref{Eq SchrodDarboux}. By $T_{1/f}$ we denote the transmutation operator for \eqref{Eq SchrodDarboux} and let $K_{1/f}$ be its integral kernel. The meaning of the indices $f$ and $1/f$ is in the following: the functions $f$ and $1/f$ uniquely identify the potentials via $q_1 = f''/f$ and $q_2 = (1/f)''/(1/f)$ and the transmutation operators via \[
T_f[1] = f\qquad\text{and}\qquad T_{1/f}[1] = \frac 1f.
\]

Consider two sequences of recursive integrals (see \cite{KrCV08}, \cite{KrPorter2010})
\begin{equation*}
X^{(0)}\equiv1,\qquad X^{(n)}(x)=n\int_{0}^{x}X^{(n-1)}(s)\left(
f^{2}(s)\right)  ^{(-1)^{n}}\,\mathrm{d}s,\qquad n=1,2,\ldots
\end{equation*}
and
\begin{equation*}
\widetilde{X}^{(0)}\equiv1,\qquad\widetilde{X}^{(n)}(x)=n\int_{0}%
^{x}\widetilde{X}^{(n-1)}(s)\left(  f^{2}(s)\right)  ^{(-1)^{n-1}}%
\,\mathrm{d}s,\qquad n=1,2,\ldots.
\end{equation*}
The families of functions $\left\{
\varphi_{k}\right\}  _{k=0}^{\infty}$ and $\left\{  \psi_{k}\right\}
_{k=0}^{\infty}$, called the systems of formal powers associated with $f$, are constructed according to the rules
\begin{equation*}
\varphi_{k}(x)=
\begin{cases}
f(x)X^{(k)}(x), & k\text{\ odd},\\
f(x)\widetilde{X}^{(k)}(x), & k\text{\ even}%
\end{cases}
\qquad\text{and}\qquad
\psi_{k}(x)=
\begin{cases}
\widetilde{X}^{(k)}(x)/f(x), & k\text{\ odd,}\\
X^{(k)}(x)/f(x), & k\text{\ even}.
\end{cases}
\end{equation*}
The following mapping properties are established in \cite{CamposKrT2012}
\begin{equation}\label{MappingThm Schrod}
T_{f}x^{k}=\varphi_{k}(x)\qquad\text{and}\qquad T_{1/f}x^{k}=\psi_{k}(x), \qquad k\in\mathbb {N}_{0}.
\end{equation}

Recall also the following definitions of generalized wave polynomials \cite{KrT2015}
\begin{align}
u_{0}&=\varphi _{0}(x),\quad u_{2n-1}(x,t)=\sum_{\text{even }k=0}^{n}\binom{n%
}{k}\varphi _{n-k}(x)t^{k},\quad u_{2n}(x,t)=\sum_{\text{odd }k=1}^{n}\binom{%
n}{k}\varphi _{n-k}(x)t^{k},  \label{um}\\
v_{0}&=\psi _{0}(x),\quad v_{2n-1}(x,t)=\sum_{\text{even }k=0}^{n}\binom{n}{k}%
\psi _{n-k}(x)t^{k},\quad v_{2n}(x,t)=\sum_{\text{odd }k=1}^{n}\binom{n}{k}%
\psi _{n-k}(x)t^{k}.  \label{vm}
\end{align}

We notice that a non-vanishing solution $(f_{0},g_{0})^{T}$ of \eqref{Equiv Dirac} corresponding to $\omega =0$ can  be chosen as
\begin{equation}\label{Eq SolHomDiracSysQzero}
f_{0}(x)=\exp\left(\int_{0}^{x}
q(s)\,ds\right) = f(x) \qquad\text{and}\qquad g_{0}(x)=\exp\left(-\int_{0}^{x}q(s)\,ds\right) = \frac{1}{f(x)}.
\end{equation}
Consider the formal powers $\left\{\Phi_{k}\right\}_{k=0}^{\infty}$ and $\left\{\Psi_{k}\right\}_{k=0}^{\infty}$ given by \eqref{Phisubk} and \eqref{Psisubk}. One can easily verify that
\begin{equation}\label{Phi and phi}
\Phi_{k}=
\begin{pmatrix}{\varphi_{k}}\\
			               {0}
\end{pmatrix}
\qquad \text{and}\qquad
\Psi_{k}=\begin{pmatrix}{0}\\
           {\psi_{k}}
	\end{pmatrix}, \qquad k=0,1,2,\ldots
\end{equation}

Consider transmutation operator $T$ for the Dirac equation \eqref{Equiv Dirac}.
Theorem \ref{Th MappingTheorem} gives us
\begin{equation}\label{Eq MappingtThcaseQ0}
T\begin{pmatrix}
                      {x^{k}}\\
                         {0}
               \end{pmatrix}
=
\begin{pmatrix}{\varphi_{k}(x)}\\
			               {0}
\end{pmatrix}
\qquad \text{and}\qquad
T
\begin{pmatrix}{0}\\
               {x^{k}}
\end{pmatrix}
=\begin{pmatrix}{0}\\
           {\psi_{k}(x)}
	\end{pmatrix}, \quad k=0,1,2,\ldots
\end{equation}

\begin{proposition}\label{Pp IntegralKernelPzero}
The integral kernel $K$ of the transmutation operator $T$ has the form
\begin{equation}\label{Sol IntegralKernelPzero}
K(x,t)=\begin{pmatrix}
                      K_{f}(x,t) & 0 \\
                               0 & K_{1/f}(x,t) \\
                    \end{pmatrix}
\end{equation}
\end{proposition}

\begin{proof} Consider the following operator $\mathcal{T}$
\[
\mathcal{T}\begin{pmatrix}y_{1}(x) \\
                          y_{2}(x)
              \end{pmatrix}= \begin{pmatrix}y_{1}(x) \\
                                            y_{2}(x)
               \end{pmatrix}+
	\int_{-x}^{x} \begin{pmatrix}K_{f}(x,t) & 0 \\
                                        0 & K_{1/f}(x,t)
                \end{pmatrix}
   \begin{pmatrix}y_{1}(t) \\
                  y_{2}(t)
   \end{pmatrix}	
	\,dt,
\]
let $\mathcal{K}$ denotes its integral kernel.

Due to \eqref{MappingThm Schrod}, \eqref{Phi and phi} and \eqref{Eq MappingtThcaseQ0} we have
\[
\mathcal{T}\begin{pmatrix}x^m \\
                                       x^n
                  \end{pmatrix}
=
\begin{pmatrix}\varphi_{m}(x) \\
                                       \psi_{n}(x)
                  \end{pmatrix}
= \Phi_m(x) + \Psi_n(x)
=T\begin{pmatrix}x^m \\
                                       x^n
                  \end{pmatrix}.
\]
Therefore
\[
 \begin{pmatrix}{0} \\
                {0}
   \end{pmatrix}
		=\left(T-\mathcal{T}\right)\begin{pmatrix}{x^{m}} \\
                                                {x^{n}}
                                  \end{pmatrix}=
	\int_{-x}^{x} \left(K(x,t)-\mathcal{K}(x,t)\right)\begin{pmatrix}{t^{m}} \\
                                                                                  {t^{n}}
                                                                   \end{pmatrix} dt.
\]
Since each row of matrix-valued function $K(x,t)-\mathcal{K}(x,t)$ is orthogonal to all vector-valued functions in the form $(x^{m}, x^{n})^{T}$, we obtain that continuous functions $K(x,t)$ and $\mathcal{K}(x,t)$ coincide.
\end{proof}

Proposition \ref{Pp IntegralKernelPzero} makes it possible to apply the analytic approximation method developed in this paper for the integral kernels $K_{f}$ and $K_{1/f}$. A question arises naturally, what is the appearance of the functions  $\mathcal {N}_{n}$ involved in Theorem \ref{Th AATOForDiracL2norm}. While it is true that the construction of Theorem \ref{Th AATOForDiracL2norm} is based on the ideas from \cite{KrT2015}, it is worth mentioning that Theorem \ref{Th AATOForDiracL2norm} does not match Theorem 5.1 from \cite{KrT2015}. In fact, Theorem \ref{Th AATOForDiracL2norm} offers a different possibility to approximate the integral kernel in \eqref{Eq VolIntOperatorForSchodinger}. Namely, the approximation of the data at $x = t$ and $x =-t$ is obtained by utilizing generalized derivatives of the systems  $\left\{\boldsymbol{c}_{n}(x)\right\}_{n=1}^{\infty}:=\left\{{u}_{2n-1}(x,x)\right\}_{n=1}^{\infty}$ and $\left\{\boldsymbol{s}_{n}(x)\right\}_{n=1}^{\infty}:=\left\{u_{2n}(x,x)\right\}_{n=1}^{\infty}$ from \cite{KrT2015}.

By induction on $n$ and using the already known relations
\begin{align*}
\partial_{x}\varphi_{k}&=\frac{f'}{f}\varphi_{k}+k\psi_{k-1},& k&=0,1,\ldots\\
f\partial_{x}\Bigl(\frac{1}{f}(x^{l}\varphi_{k})\Bigr)&=lx^{l-1}\varphi_{k}+kx^{l}\psi_{k-1},& k&=0,1,\ldots \ \text{and}\  l\geq 0.
\end{align*}
we obtain
\begin{equation}
\mathcal {N}_{n}(x)=\frac{1}{n+1}f\partial\frac{1}{f}
                                                      \begin{pmatrix}    {0}       &{\boldsymbol{c}_{n+1}}\\
                                                            {-\boldsymbol{s}_{n+1}}&{0}
			                                                 \end{pmatrix}.\label{Eq NnPzero}
\end{equation}
Hence the minimization problem \eqref{LeastSquaresMin} reduces to two independent problems
\begin{equation}\label{Approx Schrod}
    \min_{a_0,\ldots,a_N}\left\| \frac{f'}{2f} -\sum_{n=0}^{N}\frac{a_{n}}{n+1} f\partial\frac{1}{f}\boldsymbol{s}_{n+1}\right\|_{L_2(0,b)}\quad\text{and}\quad
    \min_{d_0,\ldots,d_N}\left\| \frac{f'}{2f} +\sum_{n=0}^{N}\frac{d_{n}}{n+1} f\partial\frac{1}{f}\boldsymbol{c}_{n+1}\right\|_{L_2(0,b)}.
\end{equation}

Then we obtain from Theorem \ref{Th AATOForDiracL2norm} the following result.
\begin{corollary}\label{Corr Schrod}
Let $q_1\in C^p[0,b]$ for some $p\in \mathbb{N}_0$.
For every $N>r$, let the coefficients $\left\{a_n\right\}_{n=0}^{N}\subset \mathbb{C}$ and $\left\{d_n\right\}_{n=0}^{N}\subset \mathbb{C}$ be obtained as the least squares solutions of \eqref{Approx Schrod} and let us define approximate kernels $K_{f,N}$ and $K_{1/f,N}$ by
\begin{equation}\label{Approx Kernels Schrod}
K_{f,N}=a_{0}u_{0}+\sum_{n=1}^{N}a_{n}u_{2n-1}+\sum_{n=1}^{N}d_{n}u_{2n}\quad\text{and}\quad
K_{1/f,N}=-d_{0}v_{0}-\sum_{n=1}^{N}a_{n}v_{2n} - \sum_{n=1}^N d_{n}v_{2n-1}.
\end{equation}
Then for every $x\in [0,b]$ the following estimates hold
\begin{equation}\label{Eq EstimatelKernerSchrod}
\left\|K_f(x,\cdot_{t})-K_{f,N}(x,\cdot_{t})\right\|_{L_{2}(-x,x)}< \frac{C}{N^{r+1}},\qquad \left\|K_{1/f}(x,\cdot_{t})-K_{1/f,N}(x,\cdot_{t})\right\|_{L_{2}(-x,x)}< \frac{C}{N^{r+1}},
\end{equation}
where the constant $C$ does not depend on $x$ and $N$.
\end{corollary}
\begin{proof}
Note that for $q_1\in C^r[0,b]$ one has $f\in C^{r+2}[0,b]$, hence the potential matrix $Q$ in the Dirac equation belongs to $C^{r+1}([0,b],\mathcal{M}_2)$. Now the result directly follows from Theorem \ref{Th AATOForDiracL2norm}.
\end{proof}

\begin{remark}
Corollary \ref{Corr Schrod} provides $L_2$ analogue of Theorem 7.1 from \cite{KrT2016}.
\end{remark}

\section{Numerical illustration}\label{Sect9}
Consider a one-dimensional Dirac equation \eqref{DiracSystem} with an initial condition
\begin{equation}\label{Eq IC}
Y(0)=\begin{pmatrix}{y_{1}(0)}\\{y_{2}(0)}\end{pmatrix}=\begin{pmatrix}{a}\\{b}\end{pmatrix},
\end{equation}
or a boundary condition
\begin{equation}\label{Eq BoundaryCDirac4}
\begin{pmatrix}{u_{11}}&{u_{12}}\\{u_{21}}&{u_{22}}\end{pmatrix}\begin{pmatrix}{y_{1}(0)}\\{y_{2}(0)}\end{pmatrix}+\begin{pmatrix}{u_{13}}&{u_{14}}\\{u_{23}}&{u_{24}}\end{pmatrix}\begin{pmatrix}{y_{1}(b)}\\{y_{2}(b)}\end{pmatrix}=\begin{pmatrix}{0}\\{0}\end{pmatrix}.
\end{equation}

Based on the results of the previous sections, we propose the following algorithm for numerical solution of initial value and spectral problems for \eqref{DiracSystem}.

\begin{enumerate}
	\item Find a non-vanishing solution $Y=(f,g)^{T}$ of the equation $B\frac{dY}{dx}+Q(x)Y=0$, see \cite[Section 2.3]{KTG1} for details.
	\item Compute the vector functions $\Phi_{k}$ and $\Psi_{k}$, $k=0,\ldots,N$ using \eqref{X0Y0}--\eqref{XtYt0} and \eqref{Phisubk}--\eqref{Psisubk}.
	\item Compute the matrix functions $\mathcal {N}_{k}$, $k=0,\ldots,N$, using \eqref{Df GWaveMatricesUn}, \eqref{Df GWaveMatricesVn} and \eqref{Eq halfSumKn}.
	\item Find coefficients $\left\{a_{n}, b_n, c_{n}, d_n\right\}_{n=0}^{N}$ as the least squares solution of \eqref{LeastSquaresMin}.
Here we would like to mention that the system of functions $\mathcal {N}_{k}$, $k=0,\ldots,N$, considered in machine precision, can be almost linearly dependent. Some regularization, e.g., Tikhonov regularization, may be helpful, see \cite[Sections 2.2 and 2.5--2.7]{Kirsch} and \cite[Section 7.5]{KKTD} for details.
\item Compute the matrix functions $\mathcal{K}_{2n}$ and $\mathcal{K}_{2n+1}$ using \eqref{Eq KnEvenPowert} and \eqref{Eq KnOddPowert} respectively.
\item Compute the approximate solutions $C_{N}$ and $S_{N}$ using \eqref{Eq ApproximateCosineSol_DiracSys1} and \eqref{Eq ApproximateSineSol_DiracSys1}.
\item An approximation to the solution of the initial value problem \eqref{Eq IC} is given by $Y_{N}(x,\lambda)=aC_{N}(x,\lambda)+bS_{N}(x,\lambda)$.
\item To solve the spectral problem defined by the boundary condition \eqref{Eq BoundaryCDirac4} one has to find zeros of an approximate  characteristic function of the problem which has the form
\begin{equation}\label{Eq CaracteristicfunctionAnalAprox}
\operatorname{det}\left(M_{N}(\lambda)\right)=0,
\end{equation}
where
\begin{equation}\label{ApproxCharEq}
M_{N}(\lambda)=
\begin{pmatrix}
{u_{11}}&{u_{12}}\\
{u_{21}}&{u_{22}}
\end{pmatrix}
+
\begin{pmatrix}
{u_{13}}&{u_{14}}\\
{u_{23}}&{u_{24}}
\end{pmatrix}
\begin{bmatrix}
{C_{N}(b,\lambda)}&{S_{N}(b,\lambda)}
\end{bmatrix}.
\end{equation}
\end{enumerate}

All the steps of the proposed algorithm can be performed numerically, there is no need to calculate the integrals involved analytically. We refer the reader to \cite{KTG1}, \cite{KrT2015} and \cite{KTNavarro} for additional implementation details. The computational time and resources required by the proposed method are very similar to those of \cite{KrT2015}.

\subsection{Example: integral kernel for one-dimensional Dirac system with Lorentz scalar potential.}
Consider the following  Dirac system with Lorentz scalar potential (Example 3.4 from \cite{KrT2014})
\begin{equation}\label{Eq DiracWithTanh}
\begin{pmatrix}{0}&{1}\\
              {-1}&{0}
\end{pmatrix}\frac{dY}{dx}+\begin{pmatrix}{0}&{\tanh(x)}\\
                             {\tanh(x)}&{0}
                           \end{pmatrix}Y=\lambda Y,
\quad
Y=\begin{pmatrix}{y_1}\\{y_2}\end{pmatrix}.
\end{equation}
A non-vanishing solution of equation \eqref{Eq DiracWithTanh} for $\lambda=0$ is given by  $y=(f,g)^{T}$, where $f(x)=\cosh x$, $g(x)=1/\cosh x = \operatorname{sech} x$.
According to Proposition \ref{Pp IntegralKernelPzero} and \cite{KrT2014}, the integral kernel $K$ for this example is known in the following form
\[
K=\begin{pmatrix}{ K_{\cosh}}&{0}\\{0}&{ K_{\operatorname{sech}}}\end{pmatrix},
\]
where
\begin{align*}
K_{\cosh}(x,t)&=-\frac{1}{2}\frac{\sqrt{x^{2}-t^{2}}I_{1}(\sqrt{x^{2}-t^{2}})}{x-t},\\
K_{\operatorname{sech}}(x,t)&=-\frac{1}{2\cosh x }\int_{-t}^{x}\left(\frac{I_{0}(\sqrt{s^{2}-t^{2}})t}{s-t} +\frac{\sqrt{s^{2}-t^{2}}I_{1}(\sqrt{s^{2}-t^{2}})}{s-t}\right)\cosh s\,ds,
\end{align*}
and $I_{0}$, $I_{1}$ are the modified Bessel functions of the first kind. Even though the integral in the expression for $K_{\operatorname{sech}}$ can not be calculated in a closed form, it can be evaluated numerically which is sufficient for comparison.
On Figure \ref{fig:Lorentz potentialEx1} we show the combined error of the approximation of $K_{\cosh}$ and $K_{\operatorname{sech}}$ by $K_{\cosh,10}$ and $K_{\operatorname{sech},10}$ respectively.
\begin{figure}[htbp]
\centering
\includegraphics[width=5.2in,height=3.3in]{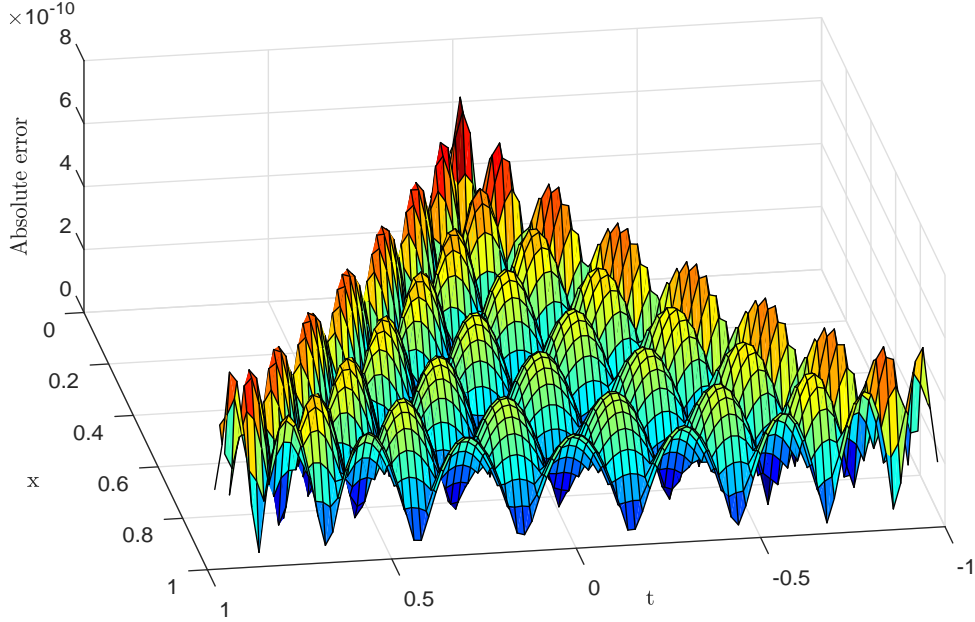}
\caption{The combined absolute error of approximation of the integral kernels $K_{\cosh}$ and $K_{\operatorname{sech}}$ obtained using the proposed method with $N=10$} \label{fig:Lorentz potentialEx1}
\end{figure}

\subsection{Example: spectral problem for a Dirac equation.}

Consider the following spectral problem (Example 3.4 from \cite{AnnabyTharwat}, see also \cite[Section 4.2]{KTG1})
\begin{equation}\label{Eq ExAnT}
\begin{pmatrix}
  {0}&{1}\\
  {-1}&{0}
\end{pmatrix}
\frac{dZ}{dx}
+\begin{pmatrix}
   {-x}&{0}\\
    {0}&{1}
	\end{pmatrix}
 Z=\lambda Z,
\quad
Z=\begin{pmatrix}
{u(x)}\\
{v(x)}
\end{pmatrix},
\quad
0\le x\le 1,									
\end{equation}
with boundary conditions
\begin{equation}\label{Eq BoundCExAnT}
u(0)=u(1)=0.
\end{equation}
Although this system does not have the form \eqref{DiracSystem} we can transform \eqref{Eq ExAnT} into \eqref{DiracSystem} using the orthogonal transformation
\begin{equation*}
Z(x)=\begin{pmatrix}
        {\cos(\varphi(x))}&{-\sin(\varphi(x))}\\
        {\sin(\varphi(x))}&{\cos(\varphi(x))}
     \end{pmatrix}
Y(x),\quad \varphi(x)=\frac{x(x-2)}{4},
\end{equation*}
see \cite{LevitanSargsjan}. It follows that the boundary value problem \eqref{Eq ExAnT}--\eqref{Eq BoundCExAnT} is equivalent to
\begin{equation*}
\begin{pmatrix}
     {0}&{1}\\
    {-1}&{0}
\end{pmatrix}
\frac{dY}{dx}+\begin{pmatrix}
                 -\frac{(x+1)}{2}\cos(2\varphi(x))&\frac{(x+1)}{2}\sin(2\varphi(x))\\
                 \frac{(x+1)}{2}\sin(2\varphi(x))&\frac{(x+1)}{2}\cos(2\varphi(x))
							\end{pmatrix}
								  Y=\lambda Y,
\end{equation*}
with the boundary conditions
\begin{equation*}
\begin{pmatrix}{1}&{0}\\
               {0}&{0}
\end{pmatrix}
Y(0)
+
\begin{pmatrix}{0}&{0}\\
               {\cos(1/4)}&{\sin(1/4)}
\end{pmatrix}
Y(1)
=0.
\end{equation*}

It can be seen from \eqref{ApproxCharEq} that an approximate characteristic equation reduces to
\begin{equation}\label{ExAnnabyCharEq}
\begin{pmatrix}{\cos(1/4)}&{\sin(1/4)}\end{pmatrix}S_{N}(1,\lambda) =0.
\end{equation}

We computed the eigenvalues $\lambda_n$ of the problem for $|n|\le 100$ taking $N=10$ for the approximate solution and compared obtained results with those from \cite{KTG1}. The computation time was less than 1 second. On Figure \ref{fig:Ex2} we present the graph of the absolute errors of eigenvalues obtained by both methods. As expected, the precision achieved by the method  from \cite{KTG1} based on the truncated SPPS representation with 100 terms and the spectral shift technique is better for the first eigenvalues, but the precision of $\lambda_n$ having $|n|\ge 12$ is better for the proposed method. Since the proposed algorithm can be changed slightly so that the eigenvalues close to 0 are taken from the SPPS method (see \cite[Example 7.5]{KrT2015} for details), more important is the accuracy of the higher-index eigenvalues. We would like to point out that errors of the computed eigenvalues do not deteriorate for large indices $n$. Also it is worth mentioning that the truncated SPPS representation with the spectral shift technique requires recalculation of the formal powers \eqref{Phisubk} and \eqref{Psisubk} for each eigenvalue, while the proposed method needs to calculate these formal powers only once. The comparison with the results reported in \cite{AnnabyTharwat} (where only 4 eigenvalues were reported) was done in \cite{KTG1}. Comparison for higher-index eigenvalues is not possible.

\begin{figure}[htbp]
\centering
\includegraphics[width=5in,height=2.6in]{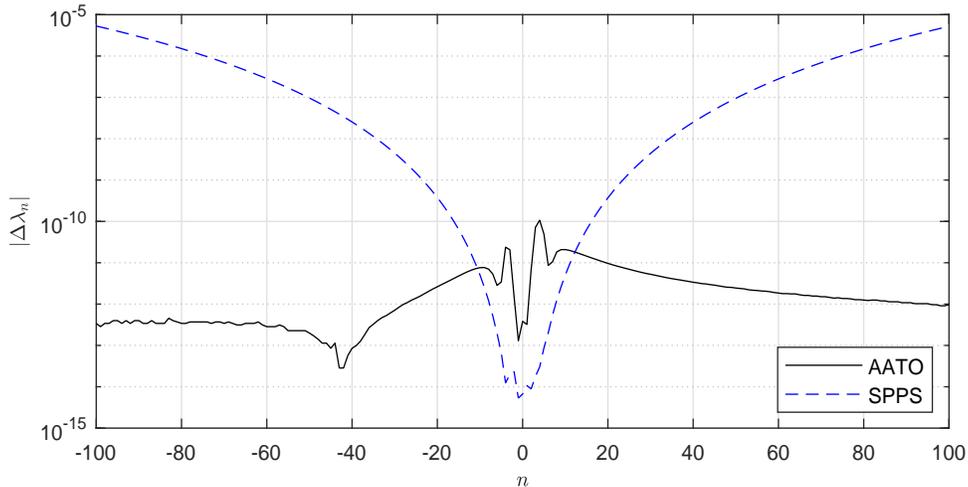}
\caption{Absolute errors of the eigenvalues of problem \eqref{Eq ExAnT}, \eqref{Eq BoundCExAnT} obtained by the proposed method (called AATO from analytic approximation of transmutation operators) with $N=10$ (solid black line) and obtained
by using
the SPPS method with $N=100$ and the spectral shift technique from \cite{KTG1} (dashed blue line).} \label{fig:Ex2}
\end{figure}

\appendix

\section{Well-posedness of the Goursat problem}
\label{Append1}
Consider the following Goursat problem in the domain $\Omega^{+}=\left\{(x,t): 0\le x\le b,\, \left|t\right|\le x\right\}$
\begin{align}
BK_{x}(x,t)+K_{t}(x,t)B &=-Q(x)K(x,t),\label{Eq ForK}\\
BK(x,x)-K(x,x)B &=E_{1}(x),\label{Eq GsatData+}\\
BK(x,-x)+K(x,-x)B&=E_{2}(x),\label{Eq GsatData-}
\end{align}
where $E_{1}$ and $E_{2}$ satisfy the compatibility conditions $E_{1}\in\mathcal{H}^{-}$ and $E_{2}\in\mathcal{H}^{+}$. Here $\mathcal{H}=L_2((0,b), \mathcal{M}_2)$ and the subspaces $\mathcal{H}^{\pm}$ are defined as in Section \ref{Sect4}.

Similar to \cite{VMchenko}, consider a change of variables $\xi=\frac{1}{2}(x+t)$, $\eta=\frac{1}{2}(x-t)$ and set $K(x,t)=H(\xi(x,t),\eta(x,t))$ in order to obtain an equivalent integral equation. It follows that
\begin{equation}\label{Eq KchainruleH}
K_x=\frac{1}{2}(H_\xi+H_\eta) \qquad \text{and} \qquad K_t=\frac{1}{2}(H_\xi-H_\eta).
\end{equation}
Substituting \eqref{Eq KchainruleH} into \eqref{Eq ForK}, multiplying on the left by $-B$ and taking into account \eqref{Eq Projectors} we get
\[
\mathcal{P}^{-}\left[H_{\xi}\right](\xi,\eta)+\mathcal{P}^{+}\left[H_{\eta}\right](\xi,\eta)=BQ(\xi+\eta)H(\xi,\eta).
\]
Also note that the left hand sides in \eqref{Eq GsatData+} and \eqref{Eq GsatData-} have the form
\begin{align}
BH(\xi,0)-H(\xi,0)B&=2B\mathcal{P}^{+}\left[H\right](\xi,0),\notag\\
BH(0,\eta)+H(0,\eta)B&=2B\mathcal{P}^{-}\left[H\right](0,\eta).\notag
\end{align}
From the above, one can see that the problem \eqref{Eq ForK}--\eqref{Eq GsatData-} becomes
\begin{align}
\mathcal{P}^{-}\left[H_{\xi}\right](\xi,\eta)+\mathcal{P}^{+}\left[H_{\eta}\right](\xi,\eta)&=BQ(\xi+\eta)H(\xi,\eta)\label{Eq KernelH},\\
\mathcal{P}^{+}\left[H\right](\xi,0)&=-\frac{1}{2}BE_{1}(\xi)\label{Eq GsatData1H},\\
\mathcal{P}^{-}\left[H\right](0,\eta)&=-\frac{1}{2}BE_{2}(\eta)\label{Eq GsatData2H},
\end{align}
in the domain
\[
\Xi^{+}=\left\{(\xi,\eta)\,|\,\,0\le \xi<b,\, 0\le \eta<b-\xi \right\}.
\]

Applying $\mathcal{P}^{-}$ on both sides of \eqref{Eq KernelH} and integrating with respect to $\xi$ yields
\[
\mathcal{P}^{-}\left[H\right](\xi,\eta)=\mathcal{P}^{-}\left[H\right](0,\eta)+\int_{0}^{\xi}\mathcal{P}^{-}\left[ BQ(u+\eta)H(u,\eta)\right]\,du.
\]
Similarly, applying $\mathcal{P}^{+}$ and integrating respect to $\eta$ yields
\[
\mathcal{P}^{+}\left[H\right](\xi,\eta)=\mathcal{P}^{+}\left[H\right](\xi,0)+\int_{0}^{\eta}\mathcal{P}^{+}\left[ BQ(\xi+v)H(\xi,v)\right]\,dv.
\]
Since the product $BQ$ belongs to $\mathcal{H}^{-}$ it is easy to check that
\begin{align*}
\mathcal{P}^{-}\left[ BQ(u+\eta)H(u,\eta)\right]&=BQ(u+\eta)\mathcal{P}^{+}\left[H\right](u,\eta),\\
\mathcal{P}^{+}\left[ BQ(\xi+v)H(\xi,v)\right]&=BQ(\xi+v)\mathcal{P}^{-}\left[H\right](\xi,v).
\end{align*}
Since $\mathcal{P}^{-}\left[H\right]+\mathcal{P}^{+}\left[H\right]=H$, it follows that
\begin{equation}\label{Eq IntegralEquationH}
H(\xi,\eta)=-\frac{1}{2}BE_{1}(\xi)-\frac{1}{2}BE_{2}(\eta)
+\int_{0}^{\xi} BQ(u+\eta)\mathcal{P}^{+}\left[H\right](u,\eta)\,du
+\int_{0}^{\eta} BQ(\xi+v)\mathcal{P}^{-}\left[H\right](\xi,v)\,dv,
\end{equation}
is an equivalent integral equation to the Goursat problem \eqref{Eq ForK}--\eqref{Eq GsatData-}.
\begin{remark}\label{Remark Eq for K}
In the case of Theorem \ref{Th DiracTransOpe}, $E_{1}(\xi)=-Q(\xi)$ and $E_{2}(\eta)=0$. Moreover, we do not need additional assumptions on the potential in order to transform problem \eqref{Eq ForK}--\eqref{Eq GsatData-} into  \eqref{Eq KernelH}--\eqref{Eq GsatData2H}, see \cite{CoxKnobel}, \cite{HrynivPronska}, \cite{LevitanSargsjan}, \cite{Yamamoto}. Thus, $K(x,t)$ is a solution of the problem \eqref{Eq ForK}--\eqref{Eq GsatData-} if and only if $H(\xi,\eta)$ is a solution of \eqref{Eq KernelH}--\eqref{Eq GsatData2H}
\end{remark}
The existence of a solution of the integral equation \eqref{Eq IntegralEquationH} can be
established by the method of successive approximations.
\begin{theorem}\label{Th ExisUniqH}
Let $Q, E_1, E_2\in L_{2}\bigl((0,b),\mathcal{M}_2\bigr)$. Then the integral equation \eqref{Eq IntegralEquationH}
has a unique solution belonging to $L_{2}\bigl(\Xi^{+},\mathcal{M}_2\bigr)$ and the following estimate holds
\begin{equation}\label{EstimateSolutionGoursat}
\| H(\xi,\eta)\|_{L_2(\Xi^+)} \le \frac{\|E_1\|_{L_2(0,b)}+\|E_2\|_{L_2(0,b)}}{2} \left(\sqrt b + \frac{b^2}2 \|Q\|_{L_2(0,b)}\exp\left( \sqrt b \|Q\|_{L_2(0,b)}\right)\right).
\end{equation}
Moreover, if the matrix-valued functions $Q$, $E_1$ and $E_2$ are continuous, then the kernel $H(\xi,\eta)$ is continuous and satisfies the inequality
\[
\left|H(\xi,\eta)\right|\le \frac{\|E_1\|_{C[0,b]}+\|E_2\|_{C[0,b]}}{2}\exp\left(b\left\|Q\right\|_{C[0,b]}\right), \qquad (\xi,\eta)\in \Xi^{+}.
\]
\end{theorem}
\begin{proof}
The proof is standard by the method of successive approximations. Let  $\left\{H_{n}\right\}_{n=0}^{\infty}$ be a sequence of matrix-valued functions given by
\begin{equation}\label{Eq SuccApproxHn}
H_{n}(\xi,\eta)=\int_{0}^{\xi} BQ(u+\eta)\mathcal{P}^{+}\left[H_{n-1}(u,\eta)\right]\,du+\int_{0}^{\eta} BQ(\xi+v)\mathcal{P}^{-}\left[H_{n-1}(\xi,v)\right]\,dv,
\end{equation}
where
\begin{equation}\label{H0term}
H_{0}(\xi,\eta)=-\frac{1}{2}BE_{1}(\xi)-\frac{1}{2}BE_{2}(\eta).
\end{equation}
Let us proceed by induction in order to get the following estimate
\begin{equation}\label{Pr IqEstimateHn}
\left|H_{n}(\xi,\eta)\right|\le \frac{C_E(\xi+\eta)^{\frac{n-1}2} \sigma^{\frac n2}(\xi+\eta)}{2\sqrt{(n-1)!n!}},
\qquad n=1,2\ldots,
\end{equation}
where
\[
C_E:=\|E_1\|_{L_2(0,b)}+\|E_2\|_{L_2(0,b)}\qquad \text{and}\qquad \sigma(\xi+\eta):=\int_{0}^{\xi+\eta}|Q(\theta)|^{2}\,d\theta.
\]

Indeed, due to the compatibility conditions of the functions $E_1$ and $E_2$ we have $\mathcal{P}^+[H_0(\xi,\eta)] = -\frac 12 B\mathcal{P}^+[E_1(\xi)]=-\frac 12 BE_1(\xi)$ and $\mathcal{P}^-[H_0(\xi,\eta)] = -\frac 12 B\mathcal{P}^-[E_2(\eta)]=-\frac 12 BE_2(\eta)$, hence
\[
\begin{split}
H_1(\xi,\eta) &= -\frac 12 \int_0^\xi BQ(u+\eta)BE_1(u)\,du - \frac 12 \int_0^\eta BQ(\xi+v) BE_2(v)\,dv \\
&= -\frac 12 \int_0^\xi Q(u+\eta)E_1(u)\,du - \frac 12 \int_0^\eta Q(\xi+v) E_2(v)\,dv,
\end{split}
\]
to obtain the last equality we used $BQ = -QB$.
It follows that
\begin{align*}
\left|H_{1}(\xi,\eta)\right|&\leq \frac{1}{2}\left(\int_{0}^{\xi}|Q(u+\eta)|^{2}\,du\right)^{1/2}\left(\int_{0}^{\xi}|E_1(u)|^{2}\,du\right)^{1/2}\\
&\quad+
\frac{1}{2}\left(\int_{0}^{\eta}|Q(\xi+v)|^{2}\,dv\right)^{1/2}\left(\int_{0}^{\eta}|E_2(v)|^{2}\,dv\right)^{1/2}\\
&\leq\frac{1}{2}\left(\|E_1\|_{L_2(0,b)}+\|E_2\|_{L_2(0,b)}\right) \left(\int_{0}^{\xi+\eta}|Q(\theta)|^{2}\,d\theta\right)^{1/2} = \frac{C_E}2 \sigma^{1/2}(\xi+\eta).
\end{align*}
Note that $\sigma'(\xi+\eta)=|Q(\xi+\eta)|^{2}$ and that $\sigma(u)$ is a monotonically increasing function. Similarly, as $H_{1}$ is the sum of two terms, one corresponding to $E_1$ and other to $E_2$, one belonging to $L_{2}^+\bigl(\Xi^{+},\mathcal{M}_2\bigr)$ and the other to $L_{2}^-\bigl(\Xi^{+},\mathcal{M}_2\bigr)$ (see Section \ref{Sect4} for definition of $^\pm$ subspaces), we see that
\begin{align*}
\displaybreak[2]
\left|H_{2}(\xi,\eta)\right|&\leq \frac{\|E_2\|}{2}\int_{0}^{\xi}|Q(u+\eta)|\sigma^{1/2}(u+\eta)du + \frac{\|E_1\|}{2}\int_{0}^{\eta}|Q(\xi+v)|\sigma^{1/2}(\xi+v)dv \\
&\le \frac{C_E}2\int_{0}^{\xi+\eta}|Q(u)|\sigma^{1/2}(u)du
\le \frac{C_E}2\left(\int_{0}^{\xi+\eta}dv\right)^{1/2}\left(\int_{0}^{\xi+\eta}|Q(u)|^{2}\sigma(u)du\right)^{1/2}\\
&=\frac{C_E}{2}\sqrt{\xi+\eta}\left(\frac{w^{2}}{2}\Bigl|_{\sigma(0)}^{\sigma(\xi+\eta)}\right)^{1/2}= \frac{C_E}{2\sqrt 2}\sqrt{\xi+\eta}\,\sigma(\xi+\eta),
\end{align*}
which coincides with the right-hand side in \eqref{Pr IqEstimateHn}.

Now we proceed by induction. Similarly,  $H_{n}$ is the sum of two terms, one corresponding to $E_1$ and other corresponding to $E_2$, belonging to $L_{2}^{\pm}\bigl(\Xi^{+},\mathcal{M}_2\bigr)$. Supposing that \eqref{Pr IqEstimateHn} holds for $n$, we get for $n+1$
\begin{align*}
\displaybreak[2]
\left|H_{n+1}(\xi,\eta)\right| &\le \frac{C_E }{2\sqrt{(n-1)!n!}} \int_0^{\xi+\eta} |Q(u)| u^{\frac{n-1}2} \sigma^{\frac n2}(u) \,du\\
\displaybreak[2]
& \le \frac{C_E }{2\sqrt{(n-1)!n!}}\left(\int_0^{\xi+\eta} u^{n-1}\,du\right)^{1/2} \left(\int_0^{\xi+\eta} |Q(u)|^2 \sigma^n(u)\,du\right)^{1/2}\\
& = \frac{C_E }{2\sqrt{(n-1)!n!}} \left(\frac{(\xi+\eta)^{n}}{n}\right)^{1/2} \left(\frac{\sigma^{n+1}(\xi+\eta)}{n+1}\right)^{1/2},
\end{align*}
which gives exactly expression \eqref{Pr IqEstimateHn} for $n+1$.

Consider the series
\begin{equation}\label{SuccessiveApproxSeries}
H(\xi,\eta):=H_0(\xi, \eta) + \sum_{n=1}^\infty H_n(\xi,\eta).
\end{equation}
Each term $H_n$, $n\ge 1$, is a continuous function possessing an estimate (due to \eqref{Pr IqEstimateHn})
\[
|H_n(\xi, \eta)|\le \frac{C_E\|Q\|_{L_2(0,b)}}{2} \frac{b^{\frac{n-1}2} \|Q\|^{n-1}_{L_2(0,b)}}{(n-1)!}.
\]
As for the first term, we obtain from \eqref{H0term} that
\[
\|H_0(\xi,\eta)\|_{L_2(\Xi^+)} \le \frac 12\left(\int_{\Xi^+}|E_1(\xi)|^2 d\xi d\eta\right)^{1/2}+\frac 12\left(\int_{\Xi^+}|E_2(\eta)|^2 d\xi d\eta\right)^{1/2} \le \frac{C_E\sqrt b}2.
\]

Finally we conclude that the series \eqref{SuccessiveApproxSeries} converges in $L_2(\Xi^+,\mathcal{M}_2)$ to the solution $H$ of \eqref{Eq IntegralEquationH} and the following estimate holds
\[
\| H(\xi,\eta)\|_{L_2(\Xi^+)} \le \frac{C_E\sqrt b}2 + \frac{C_E b^2}4 \|Q\|_{L_2(0,b)}\exp\left( \sqrt b \|Q\|_{L_2(0,b)}\right),
\]
where we used that the area of $\Xi^+$ is equal to $b^2/2$. The same estimate proves the uniqueness of the solution $H$ (assuming there are two solutions, their difference satisfies \eqref{Eq IntegralEquationH} having $E_1=E_2\equiv 0$).

For the case when the functions $Q$, $E_1$ and $E_2$ are continuous, note that the term $H_0$ is also a continuous function. Let us denote for brevity $\|\cdot\|$ the $C\bigl([0,b],\mathcal{M}_2\bigr)$ norm. Then similarly to \eqref{Pr IqEstimateHn} we obtain that
\[
|H_n(\xi,\eta)|\le \frac{\|E_1\|+\|E_2\|}{2} \frac{\|Q\|^n (\xi+\eta)^n}{n!},\qquad n=0,1,\ldots
\]
and that the solution $H$ satisfies
\[
|H(\xi,\eta)|\le \frac{\|E_1\|+\|E_2\|}{2} \exp\left(b\|Q\|\right).\qedhere
\]
\end{proof}

Now we can present proof of Proposition \ref{Pp EstimateInL2x-Norm}.
\begin{proof}[Proof of Proposition \ref{Pp EstimateInL2x-Norm}]
We have from \eqref{SuccessiveApproxSeries} and \eqref{H0term} that
\begin{equation*}\label{K via H via series}
    K(x,t) = H\left(\frac{x+t}2, \frac{x-t}2\right) = -\frac 12 BE_1\left(\frac{x+t}2\right) - \frac 12 BE_2\left(\frac{x-t}2\right) + \sum_{n=1}^\infty H_n\left(\frac{x+t}2, \frac{x-t}2\right).
\end{equation*}
Hence
\[
\begin{split}
\|K(x,&\cdot)\|_{L_2(-x,x)} \le \frac 12\left(\int_{-x}^x \left|E_1\left(\frac{x+t}2\right)\right|^2 dt\right)^{1/2} +\frac 12\left(\int_{-x}^x \left|E_2\left(\frac{x-t}2\right)\right|^2 dt\right)^{1/2} \\
& \quad + \sum_{n=1}^\infty \left(\int_{-x}^x \left|H_n\left(\frac{x+t}2, \frac{x-t}2\right)\right|^2 dt\right)^{1/2}\\
& =\frac 12\left(2\int_{0}^x \left|E_1(z)\right|^2 dz\right)^{1/2}+\frac 12\left(2\int_{0}^x \left|E_2(z)\right|^2 dz\right)^{1/2}+
\sum_{n=1}^\infty \left( \int_{-x}^x \frac{C_E^2 x^{n-1} \sigma^n(x)}{4 (n-1)! n!}dt \right)^{1/2}\\
&\le \frac{C_E}{\sqrt 2}\left( 1+\sum_{n=1}^\infty \frac{x^{n/2}\sigma^{n/2}(x)}{(n-1)!}\right) = \frac{C_E}{\sqrt 2}\left(1+ \sqrt{x\sigma(x)} \exp(\sqrt{ x \sigma(x)})\right),
\end{split}
\]
from which, recalling that $\sqrt{\sigma(x)}\le \|Q\|_{L_2(0,b)}$, the statement follows .
\end{proof}

\section{Smoothness of the integral kernel $\boldsymbol{K}$.}
\label{Append2}

\begin{proposition}\label{Prop smoothness K}
Let $Q\in C^r\bigl([0,b],\mathcal{M}_2\bigr)$ for some $r\in \mathbb{N}_0$. Then the integral kernel $K$ satisfies
\[
K\in C^r\bigl(\Omega^+, \mathcal{M}_2\bigr).
\]
\end{proposition}

\begin{proof}
According to Remark \ref{Remark Eq for K} the function $H(\xi,\eta) = K(\xi+\eta, \xi-\eta)$ satisfies the following integral equation
\begin{equation}\label{Eq IntegralForH}
H(\xi,\eta)=\frac{1}{2}BQ(\xi)+\int_{0}^{\xi} BQ(u+\eta)\mathcal{P}^{+}\left[H(u,\eta)\right]\,du+\int_{0}^{\eta} BQ(\xi+v)\mathcal{P}^{-}\left[H(\xi,v)\right]\,dv.
\end{equation}

The proof is by induction on $r$. In fact, we need only consider $r\ge 1$.

If $Q$ is continuously differentiable, one can deduce from \eqref{H0term} and \eqref{Eq SuccApproxHn} that the matrix-valued functions $H_n$ are differentiable and as a consequence, $H(\xi,\eta)$  can be differentiated with respect to both variables (we left the details to the reader).
Hence, differentiating and using \eqref{Eq KernelH} and integration by parts leads to
\begin{align}
H_{\xi}(\xi,\eta)&= \frac 12 BQ'(\xi) + BQ(\xi+\eta)\mathcal{P}^+ H(\xi,\eta) + \int_0^\eta BQ'(\xi+v)\mathcal{P}^-[H(\xi,v)]dv \nonumber\\
 &\quad+ \int_0^\eta BQ(\xi+v)\mathcal{P}^-[H_\xi(\xi,v)]dv \nonumber\\
\displaybreak[2]
&=\frac 12 BQ'(\xi) + BQ(\xi+\eta)\mathcal{P}^+ H(\xi,\eta) + \int_0^\eta BQ'(\xi+v)\mathcal{P}^-[H(\xi,v)]dv \nonumber\\
\displaybreak[2]
 &\quad+\int_0^\eta BQ(\xi+v)BQ(\xi+v)H(\xi,v)dv -\int_0^\eta BQ(\xi+v)\mathcal{P}^-[H_v(\xi,v)]dv \nonumber\\
\displaybreak[2]
&=\frac 12 BQ'(\xi) + BQ(\xi+\eta)\mathcal{P}^+ H(\xi,\eta) + \int_0^\eta BQ'(\xi+v)\mathcal{P}^-[H(\xi,v)]dv \nonumber\\
\displaybreak[2]
&\quad+\int_0^\eta Q^2(\xi+v)H(\xi,v)dv - \left.BQ(\xi+v)\mathcal{P}^+[H(\xi,v)]\right|_{v=0}^\eta + \int_0^\eta BQ'(\xi+v) \mathcal{P}^+[H(\xi,v)]dv \nonumber\\
&=\frac{1}{2}BQ'(\xi)+\frac{1}{2}Q^{2}(\xi)+\int_{0}^{\eta} \left(BQ'(\xi+v)+Q^{2}(\xi+v)\right)H(\xi,v)\,dv,\label{DerivativeRespecXi}
\end{align}
where we used \eqref{Eq GsatData1H}. Similarly,
\begin{equation}\label{DerivativeRespecEta}
H_{\eta}(\xi,\eta)=\int_{0}^{\xi} \left(BQ'(u+\eta)+Q^{2}(u+\eta)\right)H(u,\eta)\,du.
\end{equation}

Assume the statement holds for some $k$, that is, there exists continuous derivatives $\partial_\xi^l\partial_\eta^{k-l}H$ for all $0\le l\le k$. The equalities \eqref{DerivativeRespecXi}--\eqref{DerivativeRespecEta} are our tools to show that $H(\xi,\eta)$ has continuous derivatives with respect to both variables up to the order $k+1\le r$.
Given that
\begin{equation*}
\partial_{\xi}^{r+1}H(\xi,\eta)=\partial_{\xi}^{r}H_{\xi}(\xi,\eta) \qquad\text{and}\qquad \partial_\xi^l\partial_{\eta}^{r+1-l}H(\xi,\eta)=\partial_\xi^l\partial_{\eta}^{r-l}H_{\eta}(\xi,\eta),\quad l\le r,
\end{equation*}
substituting the right-hand sides of \eqref{DerivativeRespecXi} and \eqref{DerivativeRespecEta} and utilizing the induction hypothesis and the condition $Q\in C^r[0,b]$ we obtain that the above expressions are well defined, which finishes the proof.
\end{proof}

\section{Cauchy problem associated with the kernel equation}
\label{Append4}
Consider the following Cauchy problem for equation \eqref{Eq IntKernel} in domain $\Omega^+$ with initial condition given at $x=b$.
\begin{equation}\label{Pb CauchyPblemEqForK2}
	\begin{cases}
        BK_{x}(x,t)+K_{t}(x,t)B =-Q(x)K(x,t),& (x,t)\in\Omega^+,\\
        K(b,t)= F(t),& t\in [-b,b],
    \end{cases}
\end{equation}
where $F\in C([-b,b],\mathcal{M}_2)$. Similar Cauchy problem was considered in \cite{LevitanSargsjan}, however requiring additional differentiability condition on $Q$ and $F$. We present the proof of the well-posedness of this problem which neither rely on transforming the problem into a system of non-homogeneous wave equations nor require differentiability of $Q$ or $F$.

We transform this problem to an equivalent integral equation similarly to Appendix \ref{Append1}. Consider a change of variables $\xi=\frac{1}{2}(x+t)$, $\eta=\frac{1}{2}(x-t)$ and set $H(\xi,\eta) = K(x,t)$. Than equation \eqref{Eq IntKernel} transforms into equation \eqref{Eq KernelH}, and the initial condition into
\begin{equation}\label{CauchyH IC}
    H\left(\frac{b+t}2,\frac{b-t}2\right) = F(t), \qquad t\in [-b,b].
\end{equation}
Applying $\mathcal{P}^{-}$ on both sides of \eqref{Eq KernelH} and integrating with respect to $\xi$ from $\xi$ to $b-\eta$ yields
\[
\int_\xi^{b-\eta} \mathcal{P}^{-}[BQ(u+\eta) H(u,\eta)]du = \mathcal{P}^-[H](b-\eta, \eta) - \mathcal{P}^-[H](\xi,\eta) =
\mathcal{P}^-[F](b-2\eta) - \mathcal{P}^-[H](\xi,\eta).
\]
Similarly, applying $\mathcal{P}^{+}$ and integrating respect to $\eta$  from $\eta$ to $b-\xi$ yields
\[
\int_\eta^{b-\xi} \mathcal{P}^{+}[BQ(\xi+v) H(\xi,v)]dv = \mathcal{P}^+[H](\xi, b-\xi) - \mathcal{P}^+[H](\xi,\eta) =
\mathcal{P}^+[F](2\xi-b) - \mathcal{P}^+[H](\xi,\eta).
\]
And similarly to \eqref{Eq IntegralEquationH} we obtain an integral equation equivalent to Cauchy problem \eqref{Pb CauchyPblemEqForK2}.
\begin{equation}\label{Pb CauchyForH}
\begin{split}
    H(\xi,\eta) &= \mathcal{P}^+[F](2\xi-b) + \mathcal{P}^-[F](b-2\eta) \\
    &\quad- \int_\xi^{b-\eta} BQ(u+\eta)\mathcal{P}^{+}[H(u,\eta)]du - \int_\eta^{b-\xi}BQ(\xi+v) \mathcal{P}^{-}[ H(\xi,v)]dv.
\end{split}
\end{equation}

\begin{theorem}\label{Thm CauchyProblem}
Let $Q\in C\bigl([0,b],\mathcal{M}_2\bigr)$ and $F\in C\bigl([-b,b],\mathcal{M}_2\bigr)$. Then the integral equation \eqref{Pb CauchyForH}
has a unique solution belonging to $C\bigl(\Xi^{+},\mathcal{M}_2\bigr)$ and the following estimate holds
\begin{equation}\label{EstimateSolutionCauchy}
\begin{split}
| H(\xi,\eta)| &\le 2\|F\|_{C[-b,b]} \exp\left( b \|Q\|_{C[0,b]}\right).
\end{split}
\end{equation}
\end{theorem}

\begin{proof}
The proof is by the successive approximations method. Let
\[
H_0(\xi,\eta) = \mathcal{P}^+[F](2\xi-b) + \mathcal{P}^-[F](b-2\eta)
\]
and
\[
H_n(\xi,\eta) = - \int_\xi^{b-\eta} BQ(u+\eta)\mathcal{P}^{+}[H_{n-1}(u,\eta)]du - \int_\eta^{b-\xi}BQ(\xi+v) \mathcal{P}^{-}[ H_{n-1}(\xi,v)]dv.
\]
Then similarly to the proof of Theorem \ref{Th ExisUniqH} one easily obtains that
\[
|H_{n}(\xi,\eta)|\le 2\|F\| \frac{\|Q\|^n (b-\xi-\eta)^n}{n!}, \qquad n=0,1,\ldots,
\]
where $\|F\|$ and $\|Q\|$ denotes uniform norm. Hence the matrix valued function $H = \sum_{n=0}^\infty H_n$ is a solution of \eqref{Pb CauchyForH} and satisfies the following estimate
\[
|H(\xi,\eta)| \le 2\|F\|\exp( b\|Q\|),\qquad (\xi,\eta)\in\Xi^+.
\]
The last estimate proves also the uniqueness of the solution.
\end{proof}

\section{Least squares solution of minimization problem \eqref{LeastSquaresMin}}
\label{Append5}
The procedure described in this appendix follows the general theory from \cite[Section 2.2]{Kirsch}.

Consider two Hilbert spaces. $X=\mathcal{M}_2^{n+1}$, consisting of columns whose entries are $2\times 2$ matrices, equipped with the scalar product
\[
\left\langle(A_0,\ldots,A_N)^T , (B_0,\ldots,B_N)\right\rangle = \sum_{n=0}^N \operatorname{tr}\bigl(A_n B_n^\ast\bigr),
\]
and $Y=L_2((0,b),\mathcal{M}_2)$. Let
\[
C = \begin{pmatrix}
\begin{pmatrix}
a_0 & b_0\\
c_0 & d_0
\end{pmatrix}\\
 \vdots \\
\begin{pmatrix}
a_N & b_N\\
c_N & d_N
\end{pmatrix}
\end{pmatrix}\qquad \text{and}\qquad N(x) =
\begin{pmatrix}
\mathcal{N}_0(x)\\
\mathcal{N}_1(x)\\
\vdots\\
\mathcal{N}_N(x)
\end{pmatrix}.
\]
Then $C$ is the solution of the minimization problem
\begin{equation}\label{LeastSquaresMin2}
\min_{Z\in X} \left\| K Z + \frac 12 Q\right\|_Y,
\end{equation}
where $K:X\to Y$ is the operator given by $K[Z](x) = N^T(x)Z$. Let $G\in Y$. Then it is easy to see that the adjoint operator $K^\ast:Y\to X$ is given by
\[
K^\ast[G] = \int_0^b \overline{ N(x)} G(x) \,dx.
\]

According to Lemma 2.10 from \cite{Kirsch} the solution of  minimization problem \eqref{LeastSquaresMin2} coincides with the solution of the normal equation $K^\ast K Z = -\frac 12 K^\ast[Q]$, or
\[
\int_0^b \overline{ N(x)} N^T(x)\, dx \cdot Z = -\frac 12 \int_0^b \overline{ N(x)}Q(x)\,dx.
\]
Let us introduce the $2(N+1)\times 2(N+1)$ matrix $A$ consisting of the $2\times 2$ blocks $\int_0^b \overline{\mathcal{N}_i(x)} \mathcal{N}_j(x)\,dx$, $i,j=0\ldots N$. Also let us denote the first column vector of the matrix $-\frac 12 \int_0^b \overline{ N(x)}Q(x)\,dx$ by $b_1$ and the second column vector by $b_2$. Then the vector $(a_0, c_0, a_1, c_1,\ldots, a_N, c_N)^T$ (the first column-vector of $C$) is the solution of the linear system $Ax = b_1$  and the vector $(b_0, d_0, b_1, d_1, \ldots, b_N, d_N)^T$ (the second column-vector of $C$) is the solution of the linear system $Ax = b_2$.

\end{document}